\documentclass[11pt]{article}

\usepackage{graphicx}

\usepackage[a4paper,left=3cm,right=3cm,top=3cm,bottom=3cm]{geometry}

\usepackage{amssymb}
\usepackage{mathtools}
\usepackage{enumitem}
\usepackage{xcolor}
\usepackage{amsthm}
\usepackage[colorlinks=true,
    linkcolor=blue,
    filecolor=magenta,      
    urlcolor=blue, citecolor=red]{hyperref}
\usepackage[capitalize]{cleveref}
\usepackage{caption}
\usepackage{subcaption}
\usepackage{authblk}

\definecolor{myBlue}{HTML}{1C18E7}
\definecolor{myRed}{HTML}{E71C18}
\definecolor{myGreen}{HTML}{0F7A11}

\usepackage{microtype}
\DeclareFontFamily{OT1}{rsfs}{}
\DeclareFontShape{OT1}{rsfs}{n}{it}{<-> rsfs10}{}
\DeclareMathAlphabet{\mathscr}{OT1}{rsfs}{n}{it}

\newtheorem{theorem}{Theorem}
\newtheorem{lemma}[theorem]{Lemma}

\newtheorem{proposition}[theorem]{Proposition}

\newtheorem{definition}[theorem]{Definition}

\theoremstyle{definition}

\newtheorem{example}[theorem]{Example}
\newtheorem{remark}[theorem]{Remark}

\usepackage[square,numbers]{natbib}
\usepackage{doi}

\usepackage{array}

\newcolumntype{C}[1]{>{\centering\let\newline\\\arraybackslash\hspace{0pt}}m{#1}}

\usepackage{multirow}

\usepackage[textsize=footnotesize]{todonotes}

\newcommand{\diff}{\mathrm{d}}

\DeclareMathOperator{\Sym}{Sym}

\title{Homotopy Methods for Convex Optimization}
\author[1,2]{Andreas Klingler \thanks{Email: \href{mailto:andreas.klingler@univie.ac.at}{andreas.klingler@univie.ac.at}}}
\author[3]{Tim Netzer}

\affil[1]{Institute for Theoretical Physics, University of Innsbruck, Austria}
\affil[2]{Faculty of Mathematics, University of Vienna, Austria}
\affil[3]{Department of Mathematics, University of Innsbruck, Austria}

\date{\today}

\usepackage{fancyhdr} 
\makeatletter
\newcommand{\runtitle}{Homotopy methods for convex optimization}
\makeatother
\ifpdf
\hypersetup{
  pdftitle={Homotopy Methods for Convex Optimization},
  pdfauthor={A. Klingler, T. Netzer}
}
\fi

\begin{document}

\maketitle

% REQUIRED
\begin{abstract}
Convex optimization encompasses a wide range of optimization problems that contain many efficiently solvable subclasses. Interior point methods are currently the state-of-the-art approach for solving such problems, particularly effective for classes like semidefinite programming, quadratic programming, and geometric programming. However, their success hinges on the construction of self-concordant barrier functions for feasible sets.

In this work, we investigate and develop a homotopy-based approach to solve convex optimization problems. While homotopy methods have been considered in optimization before, their potential for general convex programs remains underexplored. This approach gradually transforms the feasible set of a trivial optimization problem into the target one while tracking solutions by solving a differential equation, in contrast to traditional central path methods. We establish a criterion that ensures that the homotopy method correctly solves the optimization problem and prove the existence of such homotopies for several important classes, including semidefinite and hyperbolic programs. Furthermore, we demonstrate that our approach numerically outperforms state-of-the-art methods in hyperbolic programming, highlighting its practical advantages.
\end{abstract}

\setcounter{tocdepth}{2}
\tableofcontents

\section{Introduction}
Optimization problems with convex constraints and a convex objective function (aka convex optimization problems) are a class of optimization problems that find applications in quantum information theory, portfolio optimization, data fitting, and many more \cite{Bo04}. The state-of-the-art methods to solve these problems are interior-point methods, which go back to Nesterov and Nemirovski \cite{Ne94} (see also \cite{Nesterov2018, Nemirovski2004, Wright1997, Renegar2001} for detailed overviews). While these are successfully applied to many sub-classes of convex optimization problems, including \emph{semidefinite programming}, \emph{quadratically constrained quadratic programs}, or \emph{geometric programming}, they don't solve arbitrary convex optimization problems efficiently. The success of these methods heavily rely on the construction of a self-concordant barrier function.

In this paper, we investigate and develop a way to solve certain convex optimization problems using a homotopy method. We apply this method to convex optimization problems with a single convexity constraint, as well as semidefinite and hyperbolic programs. The presented method, as all homotopy methods, is based on the following general idea:
\begin{quote}
    \emph{Homotopically change a problem with obvious solutions into the target problem and follow the path of solutions along the homotopy.}
\end{quote}

More precisely, we start from a convex optimization problem with an obvious solution (i.e.\ analytically solvable) and homotopically change it to the target problem (see \cref{fig:introduction}).  In contrast to some other homotopy methods, we propose to change the feasible set of the optimization problem, instead of changing the objective function. We then show that the path of optimal solutions is the unique solution of a system of differential equations, and we numerically solve this system to obtain a solution of the optimization problem.
This approach differs fundamentally from traditional central path methods, where solutions are obtained by solving a sequence of optimization problems at each step. Instead, we derive the entire solution path as the unique trajectory of a differential equation, providing a continuous and systematic way to reach the optimal solution without requiring iterative optimization at each stage.

\begin{figure}[h]
\centering
    \begin{tikzpicture}[scale = 0.7]

\draw (1.5, 1.5) [fill=myRed!40!white] circle (0.8cm);

\draw[myBlue] (0.94,5) -- (2.94,0);

\node[myBlue] at (1.5, 5.2) {$\mathbf{b}$};

\draw[-stealth, myBlue] (1.12,4.5) -- (1.62,4.7);

%Construction
%\draw[myBlue] (2.44,5) -- (4.44,0);
%\draw[fill=black] (3.12,3.25) circle (0.8mm);

%Construction Middle
%\draw[gray] (0.5,0.5) to[out=-5,in=210] (2.5,1) to[out=30, in=-30] (2.4,3) to[out=150,in=70] (0.4,1.8) to[out=250,in=175] (0.5,0.5);

%\draw[myBlue] (1.85,5) -- (3.85,0);
%\draw[fill=black] (2.82,2.52) circle (0.8mm);

%Path of optimal solutions
\draw[gray!70!black] (2.22,1.82) to[out=25, in=250] (2.82,2.52) to[out=70, in=200] (3.12,3.25);

\draw[dashed,draw=gray!70!white] (0,0) -- (3,1) to[out=30, in=50] (2,4) -- (0,2) -- (0,0);

\draw[fill=black] (2.22,1.82) circle (0.8mm);

\node at (2.8, 1.6) {$\mathbf{a}_0$};

\draw[-stealth, thick] (4,2) -- (5.5,2);

\begin{scope}[xshift=6cm]

\node[myBlue] at (2.38, 5.2) {$\mathbf{b}$};

\draw[-stealth, myBlue] (2.04,4.5) -- (2.52,4.7);

%Construction
%\draw[myBlue] (2.44,5) -- (4.44,0);
%\draw[fill=black] (3.12,3.25) circle (0.8mm);

%Construction Middle
\draw[fill=myRed!40!white] (0.5,0.5) to[out=-5,in=210] (2.5,1) to[out=30, in=-30] (2.4,3) to[out=150,in=70] (0.4,1.8) to[out=250,in=175] (0.5,0.5);

\draw (1.5, 1.5) [dashed,draw=gray!70!myRed] circle (0.8cm);

%Path of optimal solutions
\draw[gray!70!black] (2.22,1.82) to[out=25, in=250] (2.82,2.52) to[out=70, in=200] (3.12,3.25);

\draw[dashed,draw=gray!70!white] (0,0) -- (3,1) to[out=30, in=50] (2,4) -- (0,2) -- (0,0);

\draw[myBlue] (1.85,5) -- (3.85,0);
\draw[fill=black] (2.82,2.52) circle (0.8mm);

%Starting point
%\draw[fill=black] (2.22,1.82) circle (0.8mm);

\node at (2.3, 2.5) {$\mathbf{a}_t$};

\draw[-stealth, thick] (4,2) -- (5.5,2);

\end{scope}

\begin{scope}[xshift=12cm]

\node[myBlue] at (2.98, 5.2) {$\mathbf{b}$};

\draw[-stealth, myBlue] (2.64,4.5) -- (3.12,4.7);

%Construction Middle
\draw[fill=myRed!40!white] (0.5,0.5) to[out=-5,in=210] (2.5,1) to[out=30, in=-30] (2.4,3) to[out=150,in=70] (0.4,1.8) to[out=250,in=175] (0.5,0.5);

\draw[fill=myRed!40!white] (0,0) -- (3,1) to[out=30, in=50] (2,4) -- (0,2) -- (0,0);

%\draw[myBlue] (1.85,5) -- (3.85,0);
%\draw[fill=black] (2.82,2.52) circle (0.8mm);

%Starting point
%\draw[fill=black] (2.22,1.82) circle (0.8mm);

%\node[myBlue] at (3, 2.3) {$\mathbf{a}_0$};

%Path of optimal solutions
\draw[gray!70!black] (2.22,1.82) to[out=25, in=250] (2.82,2.52) to[out=70, in=200] (3.12,3.25);

\draw[myBlue] (2.44,5) -- (4.44,0);
\draw[fill=black] (3.12,3.25) circle (0.8mm);

\draw (1.5, 1.5) [dashed,draw=gray!70!myRed] circle (0.8cm);

\node at (3.6, 3.6) {$\mathbf{a}_1$};

\end{scope}
\end{tikzpicture}
    \caption{Visualization of the method presented in this paper. To solve the convex optimization problem on the right (i.e. maximizing $\mathbf{x} \mapsto \mathbf{b}^t \mathbf{x}$ over a convex set), one instead solves a trivial convex optimization problem (left) and tracks the path of optimal solutions via solving a differential equation.}
    \label{fig:introduction}
\end{figure}
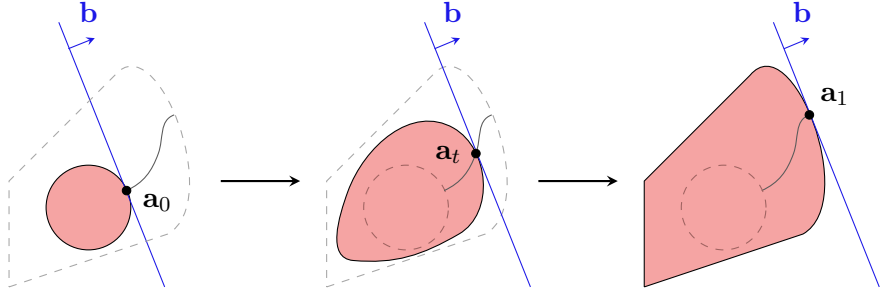

\paragraph{Related work} The homotopy method presented in this paper is deeply inspired by the principles of \emph{numerical algebraic geometry}, as developed by Sommese and Wampler \cite{So96} (see \cite{Ha17, So05} for an overview). In this field, the primary goal is to solve a system of polynomial equations,
$$p_1(\mathbf{z}) = 0, \ldots, p_k(\mathbf{z}) = 0,$$  
for multiple variables \(\mathbf{z} \in \mathbb{C}^n\) with a finite number of solutions. The approach begins by constructing an initial system of polynomial equations,
$$g_1(\mathbf{z}) = 0, \ldots, g_k(\mathbf{z}) = 0,$$
whose solutions are trivial, such as the unit polynomials $g_i(\mathbf{z}) = z_i^d$. The solutions of this system are then continuously tracked through the homotopy
$$(1-t) g_i(\mathbf{z}) + t p_i(\mathbf{z}) = 0, \quad \text{for every } i \in \{1, \ldots, k\}.$$  
In essence, the homotopy method involves gradually varying the parameter $t$ from $0$ to $1$, while following the solution path by solving an associated differential equation.

In \cite{Ha21, He21}, a homotopy approach was developed to solve semidefinite problems using the standard barrier-function relaxation. This method transforms problems lacking strict feasibility into a homotopy of problems that possess strict feasibility. However, these homotopies do not originate from a trivial initial optimization problem.

On the other hand, \cite{Na86, Na91} introduced the homotopy method for linear programs, aiming to transform a trivial problem into the target problem. In this approach, the homotopy is applied to the objective functions.

Furthermore, homotopic approach have also been applied in other contexts of linear and non-linear optimization \cite{So88, Ko91, Zh01}.

\paragraph{Main contributions} To the best of our knowledge, our approach is the first to provide a general homotopy on feasible sets of convex optimization, connecting a trivial with a very general problem, and with proven convergence under mild conditions  (see \cref{prop:uniqueness} for the main result of this paper). We prove that our approach applies successfully in the following cases:   
\begin{enumerate}[label=(\roman*)]
 \item Semidefinite and Hyperbolic programming.
    \item Optimization problems with a linear objective and one convexity constraint.
\end{enumerate}
We further benchmark the homotopy method in the following four examples:
\begin{enumerate}[label=(\alph*)]
     \item\label{ex:4} Optimizing on the hyperbolicity cone of elementary symmetric polynomials.
    \item\label{ex:1} Optimizing a linear objective function over a $k$-ellipsoid.
    \item\label{ex:2} Optimizing a linear objective function over the $p$-norm ball.
    \item\label{ex:3} Geometric programming with a single constraint.
\end{enumerate}
In particular, we show that the homotopy method outperforms the lifting techniques to semidefinite programs in \ref{ex:4}, \ref{ex:1}, and \ref{ex:2}.

\paragraph{Overview} This paper is structured as follows: In \cref{sec:diff_eq} we introduce the notion of a homotopy between optimization problems, and then show under which conditions the induced differential equation of optimal solutions has a unique solution.  In \cref{sec:examples} we apply our homotopy method to two optimization classes: hyperbolic programs (which include semidefinite programs) and convex optimization problems with a single constraint function. In \cref{sec:num_examples} we provide some numerical examples that compare our method to existing methods. In \cref{sec:outlook} we present open questions and further directions. Finally,
Appendix \ref{app:boundedness_sk} contains the proof of boundedness of rigidly convex sets defined by elementary symmetric polynomials, used in \cref{sec:num_examples}.

\section{Homotopies Between Convex Optimization Problems}\label{sec:diff_eq}

In the following, we introduce the optimization  problem we want to solve. Moreover, we introduce the notion of a homotopy between optimization problems and the ordinary differential equation for the optimal points. This section is structured as follows:
\begin{itemize}
    \item In \cref{ssec:problemStatement} we introduce the notion of a homotopy between optimization problems.
    \item In \cref{ssec:DGL_derivation} we derive the initial value problem for stationary points of the optimization problem.
    \item In \cref{ssec:uniqueness_solution} we prove that the solution of the initial value problem gives rise to the optimal value of the optimization problem. This gives rise to the main theoretical statement of this paper in \cref{prop:uniqueness}.
\end{itemize}

\subsection{Homotopies of Convex Sets and Problem Statement}\label{ssec:problemStatement}

Let $f\colon\mathbb{R}^n \to \mathbb{R}$ be a linear function, and let $C \subseteq \mathbb{R}^n$ be a compact convex set. Our objective is to find an optimal point (and thus the optimal value) of the optimization problem 
$$     \max_{\mathbf{x} \in C} \ f(\mathbf{x}).  $$
We do this by continuously deforming the feasible set from an easy one into $C$, and keeping track of an optimal solution of the corresponding optimization problems along the way. Note that this construction can be generalized  to concave functions $f$,  by introducing a homotopy of objective functions. In \cref{rem:f_homotopy} we elaborate on this more general case.

To be able to define and work with the homotopies, we will assume that the feasible sets along the homotopy (excluding the target feasible set) come with a description as follows:

\begin{definition}[Smooth description] A \emph{smooth description}
for the compact convex set  $C \subseteq \mathbb{R}^n$ is a continuously differentiable function $p\colon \mathbb{R}^n \to \mathbb{R}$ with
$$p(\mathbf x) = 0 \ \mbox{ and } \ \nabla p(\mathbf x) \neq 0 \ \mbox{ for all } \mathbf x\in\partial C.$$
\end{definition}
Capturing the boundary of $C$ is enough, since the maximum of $f$ is always attained on $\partial C$. However, note that we do not require $p(\mathbf{x}) = 0$ to imply $\mathbf{x} \in \partial C$ (see for example real zero polynomials in \cref{ssec:homotopy_RZ}). In particular, the set $C$ is in general  not uniquely determined by a smooth description, so we still have to keep track of it.
Also note that the existence of a smooth description imposes conditions on the set $C$. For example, it implies that $C$ is smooth, in the sense that each boundary point admits a unique supporting hyperplane. 

We say that a twice continuously differentiable function $p\colon \mathbb{R}^n \to \mathbb{R}$ is \emph{concave} at a point $\mathbf{x} \in \mathbb{R}^n$, if its Hessian matrix $\nabla^2 p(\mathbf{x})$ is negative semidefinite, i.e.\ all its eigenvalues are non-positive. It is \emph{strictly concave} if $\nabla^2 p(\mathbf{x})$ is negative definite, i.e.\ all eigenvalues are strictly negative.
We say that $p$ is \emph{strictly quasi-concave} at  $\mathbf{x} \in \mathbb{R}^n$, if the quadratic form $\mathbf{v} \mapsto \mathbf{v}^t \nabla^2 p(\mathbf{x}) \mathbf{v}$ is negative definite on the orthogonal complement of $\nabla p(\mathbf{x})$. In other words,
\begin{equation}\label{eq:quasi-concave}
    \mathbf{v}^t \nabla p(\mathbf{x}) = 0 \quad \Longrightarrow \quad \mathbf{v}^t \nabla^2 p(\mathbf{x}) \mathbf{v} < 0.
\end{equation}
for every $\mathbf{v} \in \mathbb{R}^n \setminus \{0\}$.\footnote{This definition is stronger than the usual definition of strict (quasi-)concavity (see for example Section 3.4.2 and Section 3.4.3 of \cite{Bo04}).  However, all (quasi-)concave functions in this paper satisfy this stronger type of (quasi-)concavity.}

If a smooth description of $C$ is strictly quasi-concave at a point $\mathbf x\in\partial C$, the point is an exposed point, i.e. the unique supporting hyperplane at $\mathbf x$ intersects $C$ only in $\mathbf x$. In particular, if this holds for all boundary points, then maximizing a linear function over $C$ has a unique solution.

Now let $\left( C_t\right)_{t\in[0,1]}$ be a   family of compact convex subsets of $\mathbb R^n.$  For our algorithm  we need a suitable  homotopy of smooth descriptions. 

\begin{definition}[Homotopy of smooth descritions]
(i) A \emph{homotopy of smooth descriptions} for the family $\left(C_t\right)_{t\in[0,1]}$  is a  continuously differentiable  function
$$p\colon [0,1] \times \mathbb{R}^n;\ (t,\mathbf{x}) \mapsto p_t(\mathbf{x})$$ such that $p_t$ provides a smooth description of $C_t$ for all $t\in[0,1)$.

(ii) The homotopy $p$ of smooth descriptions is called \emph{regular}, if for all $t^*\in [0,1],$ $\varepsilon>0,$ and all continuously differentiable functions  $\mathbf{a}\colon [t^*-\varepsilon,t^*+\varepsilon] \to \mathbb{R}^n$ with
 $\mathbf{a}_{t^*} \in \partial C_{t^*}$ and $p_t(\mathbf{a}_t) = 0 $ for all $t\in [t^*-\varepsilon,t^*+\varepsilon],$ we have  $\mathbf{a}_t \in \partial C_t$ for every $t \in [t^*-\varepsilon,t^*+\varepsilon]$. Note that we write $\mathbf a_t$ for $\mathbf a(t)$.
\end{definition}

Note that the descriptions of $C_t$ only needs to be smooth for $t < 1$.
Although the target feasible set $C_1$ does not need to be smooth (i.e. it has it has distinct vertices or kinks like in \cref{fig:introduction}), there exist homotopies of smooth descriptions with $p_1$ providing a description of $C_1$. 
 
 \begin{remark} 
 In words, regularity means that we cannot leave the boundary of the sets over time, without the functions $p_t$ noticing.
 Not every homotopy of smooth descriptions is regular.  As a counterexample one can construct  a homotopy of smooth descriptions for a family that switches in between two disjoint balls.

But under a certain continuity condition on the family $\left(C_t\right)_{t\in [0,1]}$, every homotopy of smooth descriptions is indeed regular. Assume, for example, that $$B:=\bigcup_{t\in [0,1)} \partial C_t \times \{t\} \subseteq \mathbb R^{n+1}$$ is an $n$-dimensional manifold (with boundary $\partial C_0$). Then at each point $(\mathbf x,t)\in B$ the gradient of $p$ is nonzero, already in space direction, since $p_t$ is a smooth description of $C_t$. Thus the implicit function theorem implies that the zero set of $p$ is the graph of a function in  $t$ and $n-1$ of the space variables, locally around $(\mathbf x,t)$. But since $B$ is an $n$-dimensional manifold contained in this zero set/graph, it  coincides with it locally.

Now for a continuous function $\mathbf  a\colon[t^*-\varepsilon,t^*+\varepsilon]\to\mathbb R^n$ with  $p_t(\mathbf a_t)=0$ for all $t$, the set $$I=\{ t\in [t^*-\varepsilon,t^*+\varepsilon]\mid \mathbf{a}_t\in \partial C_t\}$$ is easily checked to be open and closed. From $\mathbf a_{t^*}\in\partial C_{t^*}$ we obtain $t^*\in I$ and thus $I=[t^*-\varepsilon,t^*+\varepsilon].$

This implies that $\mathbf{a}_t \in \partial C_t$ for every $t < 1$. $\mathbf{a}_1 \in \partial C_1$ follows by continuity of $t \mapsto \mathbf{a}_t$ and $t \mapsto \mathbf{p}_t(\mathbf{x})$.
\end{remark}

In the context of the homotopy method, we now consider for every $t \in [0,1]$ the optimization problem
\begin{equation}
\label{eq:C_optimization}
    \max_{\mathbf{x} \in C_t} \  f(\mathbf{x})
\end{equation}
and we denote an optimal point by $\mathbf{a}_t \in \mathbb{R}^n$. 
To solve the optimization problems \eqref{eq:C_optimization}, we will make use of the fact that the condition $p_t(\mathbf{x}) = 0$ is necessary to have that $\mathbf{x} \in \partial C_t$. For this reason, we instead consider the following optimization problem
\begin{align}
\label{eq:root_optimization}
\begin{split}
    \max_{\mathbf{x} \in \mathbb{R}^n} \ & f(\mathbf{x})\\
    \text{ subject to } & p_t(\mathbf{x}) = 0
\end{split}
\end{align}
and use that fact that any optimal point of  \eqref{eq:C_optimization} is a stationary point of \eqref{eq:root_optimization}. We will then obtain a differential equation whose solution is a function $\mathbf{a}\colon [0,1] \to \mathbb{R}^n$ of stationary points of \eqref{eq:root_optimization}. We prove in \cref{prop:uniqueness} that this path of stationary points are the optimal points of \eqref{eq:C_optimization}.

\subsection{The Differential Equation for Stationary Points}\label{ssec:DGL_derivation}

We now derive a condition for stationary points of the optimization problem \eqref{eq:root_optimization} for every $t \in [0,1]$ via Lagrange multipliers. Afterwards, we translate these conditions into an initial value problem (see \cref{eq:diff_eq}).

Using Lagrange multipliers, one obtains that $\mathbf{a}_t$ is a \emph{stationary point}\footnote{i.e.\ $f(\mathbf{a}_t)$ is either a local minimum, local maximum or a saddle point when considered on the set $\{\mathbf{x} \in \mathbb{R}^n \mid p_t(\mathbf{x}) = 0\}$.} of the problem \eqref{eq:root_optimization} at time $t$ (assuming $\nabla p_t(\mathbf a_t)\neq 0$) if and only if 
\begin{enumerate}[label=($\text{C}_{\arabic*}$)]
    \item\label{cond:lagrange1} $\nabla f(\mathbf{a}_t) = \lambda_t \nabla p_t(\mathbf{a}_t)$
    \item\label{cond:lagrange2} $p_t(\mathbf{a}_t) = 0$
\end{enumerate}
where $\lambda_t \in \mathbb{R}$ is the Lagrange multiplier for the constraint $p_t$.

In the following, we translate these two conditions into an initial value problem, whose solution is a path $\mathbf{a}\colon [0,1] \to \mathbb{R}^n; t \mapsto \mathbf{a}_t$ where $\mathbf{a}_t$ is a stationary point of the constrained optimization problem at time $t$.

We first formulate an equivalent condition for \ref{cond:lagrange1} using the following Lemma.

\begin{lemma}\label{lem:minor_det}
Let $\mathbf{v}, \mathbf{w} \in \mathbb{R}^n \setminus \{0\}$. If $v_k \neq 0$, we have that $\mathbf{v} = \lambda \mathbf{w}$ for some $\lambda \in \mathbb{R}$ if and only if
\begin{equation}\label{eq:minor_det}
v_k w_{i} - v_{i} w_k = 0
\end{equation}
for every $i \in \{1, \ldots, n\}$.
\end{lemma}
\begin{proof}
The only if direction is immediate. So let $\mathbf{v}, \mathbf{w} \in \mathbb{R}^n \setminus \{0\}$ such that \eqref{eq:minor_det} holds. Note that also $w_k\neq 0$ holds, since otherwise we would get $\mathbf w=0.$ Now we define $\lambda$ via $v_k = \lambda w_k$. We now show that $v_i = \lambda w_i$ for every $i$. We have two cases:
\begin{enumerate}[label=(\roman*)]
    \item $w_i \neq 0$. Then $v_i, w_k \neq 0$ and we have that
    $$\frac{v_i}{w_i} = \frac{v_k}{w_k} = \lambda$$
    \item $w_i = 0$. Then $v_i w_k = 0$. Since $w_k \neq 0$, we have that $v_i = 0$ which implies that $v_i = \lambda w_i$.
\end{enumerate}
\end{proof}

By \cref{lem:minor_det}, \ref{cond:lagrange1} is equivalent to 
\begin{equation}\label{eq:parallel_condition}
    \mathcal{Q}_{k,i}(\mathbf{a}_t, t) \coloneqq \partial_k f(\mathbf{a}_t) \partial_{i} p_t(\mathbf{a}_t) -  \partial_{i} f(\mathbf{a}_t) \partial_{k} p_t(\mathbf{a}_t) = 0
\end{equation}
for $i \in \{1, \ldots, k-1, k+1, \ldots, n\}$, where $k$ is fixed with $\partial_{k} f(\mathbf{a}_t) \neq 0$.\footnote{Note that for a linear function $f \neq 0$, we can choose $k$ which satisfies $\mathbf{e}_k \notin \ker f$.} Note, that we use the notation $\partial_i h(\mathbf{x}) \coloneqq (\nabla h(\mathbf{x}))_i = \frac{\partial}{\partial x_i} h(\mathbf{x})$.

Differentiating \eqref{eq:parallel_condition} and \ref{cond:lagrange2}, we obtain an initial value problem for \emph{stationary points} $\mathbf{a}\colon  [0,1] \to \mathbb{R}^n$:
\begin{align*}
\mathbf{a}_0 &\in \mathbb{R}^n \textrm{ is an optimal point of \eqref{eq:C_optimization} at } t = 0\\ 
0 &= \big[ \nabla \mathcal{Q}_{k,i}(\mathbf{a}_t,t) \big]^t \cdot \frac{\diff \mathbf{a}_t}{\diff t}   + \partial_t \mathcal{Q}_{k,i}(\mathbf{a}_t,t) \quad \text{ for } i \in \{1, \ldots, k-1, k+1, \ldots, n\}\\
0 &= \big[\nabla p_t(\mathbf{a}_t)\big]^t \cdot \frac{\diff \mathbf{a}_t}{\diff t}  + \partial_t p_t(\mathbf{a}) 
\end{align*}
where $k$ is chosen such that $\partial_k f \neq 0$.
In summary, for every solution $\mathbf{a}\colon t \mapsto \mathbf{a}_t$ we have that $\mathbf{a}_t$ is a stationary point of the constraint optimization problem of $p_t$ and $f$. Equivalently, the differential equation reads as
\begin{equation}\label{eq:diff_eq}
K(\mathbf{a}_t, t) \cdot \frac{\diff \mathbf{a}_t}{\diff t} + m(\mathbf{a}_t, t) = 0
\end{equation}
where
\begin{equation}\label{eq:def_K_and_m}
K(\mathbf{x},t) = \begin{bmatrix} \nabla p_t(\mathbf{x})^t \\ \nabla \mathcal{Q}_{k,1}(\mathbf{x},t)^t \\ \vdots \\ \nabla \mathcal{Q}_{k,k-1}(\mathbf{x},t)^t \\ \nabla \mathcal{Q}_{k,k+1}(\mathbf{x},t)^t \\ \vdots \\ \nabla \mathcal{Q}_{k,n}(\mathbf{x},t)^t \end{bmatrix} \quad  \text{ and } \quad m(\mathbf{x},t) = \begin{bmatrix}\partial_t p_t(\mathbf{x})\\ \partial_t \mathcal{Q}_{k,1}(\mathbf{x},t)\\ \vdots \\ \partial_t \mathcal{Q}_{k,k-1}(\mathbf{x},t) \\ \partial_t \mathcal{Q}_{k,k+1}(\mathbf{x},t) \\ \vdots \\ \partial_t \mathcal{Q}_{k,n}(\mathbf{x},t) \end{bmatrix}.
\end{equation}

\subsection{Existence of a Unique Path of Optimal Solutions}\label{ssec:uniqueness_solution}

We now prove that there is a unique solution for the differential equation \eqref{eq:diff_eq}, if $\mathbf{x} \mapsto p_t(\mathbf{x})$ is a strictly quasi-concave function for every $t \in [0,1)$.\footnote{This means that the target convex set $C_1$ does not need to be strictly quasi-convex for the success of the method.} The proof is split into two parts:
\begin{enumerate}[label=(\roman*)]
    \item In \cref{lem:invertability} we show that $K(\mathbf{a}_t, t)$ is invertible, if $\mathbf{a}_t$ is a stationary point at time $t$.
    \item In \cref{prop:uniqueness}, we prove the solution is unique if $K(\mathbf{a}_t,t)$ is invertible for every $t$.
\end{enumerate}

We now show that $K(\mathbf{a}_t, t)$ is invertible under certain assumptions on $p_t$.

\begin{lemma}\label{lem:invertability} Let $\mathbf{a}_t$ be a stationary point of the constrained problem \eqref{eq:root_optimization} at time $t$. If $p_t$ is strictly quasi-concave at $\mathbf a_t$, then $K(\mathbf{a}_t,t)$ is invertible.
\end{lemma}

\begin{proof}
We show that for every $0\neq \mathbf{v} \in \mathbb{R}^n$, we have that
$$\mathbf{v}^t \nabla p_t(\mathbf{a}_t) \neq 0 \quad \text{ or } \quad \mathbf{v}^t \nabla \mathcal{Q}_{k,i}(\mathbf{a}_t, t) \neq 0 \text{ for some } i \in \{1, \ldots, k-1, k+1, \ldots, n\}.$$
So assume that $\mathbf{v}^t \nabla p_t(\mathbf{a}_t) = 0$. Since $\mathbf a_t$ is a stationary point we have $\mathbf{v}^t \nabla f(\mathbf{a}_t) = 0$, and  by strict quasi-concavity of $p_t$ that $\mathbf{v}^t \nabla^2 p_t(\mathbf{a}_t) \mathbf{v} \neq 0.$
For every $i \in \{1, \ldots, k-1, k+1, \ldots, n\}$, we can rewrite
$$ \nabla \mathcal{Q}_{k,i}(\mathbf{a}_t, t)^t \mathbf{v} = \partial_k f(\mathbf{a}_t) \big(\nabla^2 p_t(\mathbf{a}_t) \mathbf{v}\big)_i - \partial_i f(\mathbf{a}_t) \big(\nabla^2 p_t(\mathbf{a}_t) \mathbf{v}\big)_k$$

By \cref{lem:minor_det}, we have that $\nabla\mathcal{Q}_{k,i}(\mathbf{a}_t,t)^t \mathbf{v} = 0$ for all $i \in \{1, \ldots, k-1, k+1, \ldots, n\}$ if and only if $\nabla^2 p_t(\mathbf{a}_t) \mathbf{v}$ is parallel to $\nabla f(\mathbf{a}_t)$. But since we have
$$\mathbf{v}^t \nabla^2 p_t(\mathbf{a}_t) \mathbf{v} \neq 0 \quad \text{ and } \quad \mathbf{v}^t \nabla f(\mathbf{a}_t) = 0$$
we get a contradiction.
\end{proof}

We are now ready to show that the initial value problem \eqref{eq:diff_eq} has a unique solution, which gives us the solutions of  \eqref{eq:C_optimization}. This is the main theoretical result of our paper.

\begin{theorem}\label{prop:uniqueness} Let $\left(C_t\right)_{t\in[0,1]}$ be a family of compact convex sets, all with non-empty interior in $\mathbb R^n$. 
Let $p$ be a homotopy of smooth descriptions satisfying the following properties:
\begin{itemize}
    \item $p$ is regular
    \item $p$ is three times continuously differentiable
    \item $p_t$ is strictly quasi-concave on $\partial C_t$ for all $t \in [0,1)$.
\end{itemize}
Then, if the initial value $\mathbf{a}_0$ is an optimal point of the optimization problem \eqref{eq:C_optimization} for $t = 0$, the initial value problem \eqref{eq:diff_eq} has a unique solution, which consists of the optimal points of the problems \eqref{eq:C_optimization}.
\end{theorem}
\begin{proof}
For $t^* \in [0,1)$ let $\mathbf{a}_{t^*} \in \partial C_{t^*}$ be the optimal point of the optimization problem \eqref{eq:C_optimization}, where uniqueness follows from strict quasi-concavity. By \cref{lem:invertability} we know that $K(\mathbf{a}_{t^*},t^*)$ is invertible. Since invertibility is an open condition, this is also the case for every element  contained in $R \coloneqq  B_\varepsilon(\mathbf{a}_{t^*})\times [t^*-\varepsilon, t^*+\varepsilon]$ for a suitable choice of $\varepsilon > 0$. Hence, we can locally define the differential equation
$$ \frac{\diff \mathbf{x}}{\diff t} = \mathcal{F}(\mathbf{x}, t) \coloneqq - K(\mathbf{x}, t)^{-1} \cdot m(\mathbf{x},t).$$
Since the entries of $m(\mathbf{x},t)$ and $K(\mathbf{x},t)^{-1}$ are still continuously differentiable, we have that $\mathcal{F}_{|_R}$ is Lipschitz-continuous on $R$. Hence there exists a unique solution $\mathbf{a}\colon [t^*-\varepsilon, t^*+\varepsilon] \to B_\varepsilon(\mathbf{a}_{t^*})$ with $\mathbf a(t^*)=\mathbf a_{t^*}$, by the Picard--Lindel\"of Theorem. 

From regularity of $p$  we obtain $\mathbf a_t\in\partial C_t$ for all $t\in[t^*-\varepsilon, t^*+\varepsilon]$. This implies that all $\mathbf a_t$ are the optimal points of \eqref{eq:C_optimization}. In fact, since all $C_t$ have nonempty interior, and $\mathbf a$ is continuous, a jump from the optimal point to the least optimal point, which is the other critical point on the boundary, is not possible. Hence, the local solutions must glue to a global solution.

This gives rise to a unique solution $\mathbf{a}\colon [0,1) \to \mathbb{R}^n$. By the boundedness and continuity of $\mathbf{a}$, there is a unique limiting point $$\mathbf{a}_1 \coloneqq \lim_{t \uparrow 1} \mathbf{a}_t$$ which is the optimal solution of the target problem.
\end{proof}

\begin{remark}\label{rem:f_homotopy}
Note that \cref{prop:uniqueness} remains valid even when the problem is generalized to allow:
\begin{enumerate}[label=(\roman*)]
    \item\label{gen_cond:1} a homotopy of objective functions.
    \item\label{gen_cond:2} concave functions satisfying $\nabla f_t \neq 0$ on $C_t$, ensuring that the optimal point lies on the boundary.
\end{enumerate}

Regarding \ref{gen_cond:1}, the construction of the differential equation follows the same process as in \cref{ssec:DGL_derivation}, with the only modifications being the replacement of $f$ by $f_t$ and the inclusion of its non-trivial time derivatives in $m(\mathbf{a}_t, t)$. Since the invertibility of $K(\mathbf{x}, t)$ remains unaffected by the choice of a homotopy for $f$, the result in \cref{prop:uniqueness} follows immediately.

Regarding \ref{gen_cond:2}, the only remaining step is to show that $K(\mathbf{a}_t,t)$ is invertible. This follows from a slight modification of the proof of \cref{lem:invertability}.
Since $\nabla f_t \neq 0$ on $C_t$, we have $\mathbf{a}_t \in \partial C_t$. By the concavity of $f_t$, it follows that
$$\nabla f_t(\mathbf{a}_t)^t (\mathbf{x}-\mathbf{a}_t) \leqslant 0, $$
for all $\mathbf{x} \in C_t$. Since $\nabla f_t(\mathbf{a}_t) \neq 0$ on $C_t$, this implies that $\lambda_t < 0$.

We now prove that for every $\mathbf{v} \in \mathbb{R}^n$, we have
$$\mathbf{v}^t \nabla p_t(\mathbf{a}_t) \neq 0 \quad \text{ or } \quad \mathbf{v}^t \nabla \mathcal{Q}_{k,i}(\mathbf{a}_t, t) \neq 0 \quad \text{ for some } i \in \{1,\ldots, n-1\}.$$
Assuming $\mathbf{v}^t \nabla p_t(\mathbf{a}_t) = 0$ implies $\mathbf{v}^t \nabla f_t(\mathbf{a}_t) = 0$. Moreover, since $\lambda_t < 0$ and $f_t$ is strictly concave, we obtain
$$\mathbf{v}^t A(\mathbf{a}_t,t) \mathbf{v} < 0,$$
where
$$A(\mathbf{a}_t,t) = \nabla^2 f_t(\mathbf{a}_t) - \lambda_t \nabla^2 p_t(\mathbf{a}_t).$$

For every $i \in \{1, \ldots, k-1, k+1, \ldots, n\}$, we rewrite
\begin{align*}
\nabla \mathcal{Q}_{k,i}(\mathbf{a}_t, t)^t \mathbf{v} = \big(A(\mathbf{a}_t,t) \mathbf{v}\big)_k \partial_i p_t(\mathbf{a}_t) - \big(A(\mathbf{a}_t,t) \mathbf{v}\big)_i \partial_k p_t(\mathbf{a}_t),
\end{align*}
where we used the relation $\partial_i f(\mathbf{a}_t) = \lambda_t \partial_i p_t(\mathbf{a}_t)$.

By \cref{lem:minor_det}, $\nabla\mathcal{Q}_{k,i}(\mathbf{a}_t,t)^t \mathbf{v} = 0$ for all $i \in \{1, \ldots, n-1\}$ if and only if $A(\mathbf{a}_t,t)\mathbf{v}$ is parallel to $\nabla p_t(\mathbf{a}_t, t)$. However, since we have
$$\mathbf{v}^t A(\mathbf{a}_t, t) \mathbf{v} < 0 \quad \text{ and } \quad \mathbf{v}^t \nabla p_t(\mathbf{a}_t) = 0,$$
this leads to a contradiction.
\end{remark}

\section{Examples of Homotopies}
\label{sec:examples}

In the following, we present two classes of examples of  homotopies that fulfill the assumptions from \cref{prop:uniqueness} above.  In \cref{ssec:homotopy_RZ} we apply our method to the class of hyperbolic and semidefinite programs. In \cref{ssec:singleConcaveConstraint} we apply our method to convex optimization problems, where the constraint is generated by a single concave function.

While the homotopy in \cref{ssec:singleConcaveConstraint} has the desired properties  for basic reasons, the homotopy for hyperbolic and real zero polynomials is based on the smoothing theory of hyperbolic polynomials which goes back to Nuij \cite{Nu68} and Renegar \cite{Re04}. For self-consistency reasons, we will state the most important results in \cref{ssec:homotopy_RZ}.

\subsection{Semidefinite and Hyperbolic Programming}\label{ssec:homotopy_RZ}

We now consider the problem of maximizing a linear function on a \emph{spectrahedron} with non-empty interior\footnote{We assume again without loss of generality that $\mathbf{0}$ is an interior point of the set of feasible points.}, i.e.
\begin{align*}
    \max_{\mathbf{x} \in \mathbb{R}^n} \ & f(\mathbf{x})\\
    \text{ subject to } \ & A(\mathbf{x}) \succcurlyeq  0
\end{align*}
where
$$A(\mathbf{x}) = I_s + x_1 A_1 + \cdots + x_n A_n$$
with $A_1, \ldots, A_n \in \Sym_{s}(\mathbb{R})$ being real symmetric matrices of size $s \times s$, $I_s$ the identity matrix of size $s$, and $X \succcurlyeq 0$ means that $X$ is positive semidefinite. This convex optimization problem is known as a \emph{semidefinite program}.
The feasible set
$$\mathcal{S}(A) \coloneqq \{\mathbf{x} \in \mathbb{R}^n\mid A(\mathbf{x}) \succcurlyeq 0 \}$$
is called a \emph{spectrahedron}.\footnote{Note that in general a spectrahedron is defined using an arbitrary symmetric matrix $A_0$ instead of $I_s$. However, every spectrahedron with non-empty interior can be represented (up to a translation) in this form (see for example  Lemma 2.12 and Corollary 2.16 in \cite{Ne23}).}

Generalizing spectrahedra to \emph{rigidly convex} sets (see \cref{def:real_zero}) leads to the notion of \emph{hyperbolic programming}. 

A hyperbolic program is a linear optimization problem over an affine slice of a hyperbolicity cone. Specifically, it is given by  
\begin{align*}
    \max_{\mathbf{x} \in \mathbb{R}^m} \ & f(\mathbf{x}) \\
    \text{subject to} \ & M \mathbf{x} = \mathbf{b}, \\
    & \mathbf{x} \in \Lambda_{++}(h),
\end{align*}  
where $\Lambda_{++}(h)$ denotes the hyperbolicity cone associated with a hyperbolic polynomial $h$ (see \cref{ssec:hyperbolicProgramming} for further details).

For a non-trivial affine slice, i.e., $M \neq 0$, hyperbolic programs can be reformulated as
\begin{align*}
    \max_{\mathbf{x} \in \mathbb{R}^n} \ & f(\mathbf{x}) \\
    \text{subject to} \ & \mathbf{x} \in \mathcal{R}(p),
\end{align*}  
where \(\mathcal{R}(p)\) is the rigidly convex set in a smaller space, defined by a real-zero polynomial $p$ (see \cref{def:real_zero}).\footnote{
Assume there exists a vector $\mathbf{e}$ in the interior of $\Lambda_{++}(h)$ such that $M\mathbf{e} = \mathbf{b}$. If $\ker(M)$ has dimension $n = m-1$, we can rotate and rescale the space so that $\mathbf{e} = \mathbf{e}_1$. This transformation gives the real zero polynomial  
$$p(\mathbf{x}) = h(1, x_1, \dots, x_n),$$  
with \(\mathcal{R}(p)\) corresponding to \(\Lambda_{++}(h) \cap \{x_0 = 1\}\) (see \cite[Section 2.3]{Ne23} and \cite[Definition 2.3.4]{Ne12}). If $\ker(M)$ has smaller dimension, this process can be iterated to obtain the formulation.}

This class of sets generalizes semidefinite programming over spectrahedra (see \cref{ex:SDP_as_RZpoly}).

Throughout this paper, we assume that the constraint $M \mathbf{x} = \mathbf{b}$ is non-trivial; otherwise, the feasible set would be $\Lambda_{++}(h)$, which is unbounded. This allows us to consider the hyperbolic program as an optimization problem over a rigidly convex set. Note that a similar situation occurs in semidefinite programming when $I_s$ is replaced by $\mathbf{0}$ in $A(\mathbf{x})$.

Rigidly convex sets arise from  so-called \emph{real zero polynomials} and share many properties with spectrahedra. The  \emph{generalized Lax conjecture} actually postulates that every rigidly convex set is a spectrahedron. While this statement is true for $n = 2$, it is open in higher dimensions \cite{Le05}.

Note that our homotopy approach is useful beyond the established methods for solving semidefinite programs:
\begin{itemize}
    \item It is not clear whether each rigidly convex set is actually a spectrahedron, so we solve a potentially larger class of problems. It seems that no efficient general methods to solve hyperbolic programs are available yet. 
    \item Even if a set is known to be a spectrahedron, it might be hard to find an explicit linear matrix inequality defining it.
    \item Even if the linear matrix inequality is known, it may be much larger than the degree of the corresponding real-zero polynomial, while we work directly with the polynomial. This behavior is on the one hand generic \cite{Rag19}, but there are also explicit examples with a proven super-polynomial overhead when described as spectrahedra \cite{Oli20}.
\end{itemize}

We will now review real zero polynomials and show that there exists a regular homotopy in the space of real zero polynomials. The special case of semidefinite programming follows as a special case (see \cref{ex:SDP_as_RZpoly}). 

\subsubsection{Rigidly Convex Sets}

In the following, we define and examine the notion of \emph{rigidly convex sets}. Moreover, we review a procedure to approximate every rigidly convex set by a smooth version thereof. Most of the content in this subsection is also part of \cite[Section 2]{Ne12}.

\begin{definition}[Real zero polynomials and rigidly convex sets]
\label{def:real_zero}
A polynomial $p\colon \mathbb{R}^n \to \mathbb{R}$  is called a \emph{real zero polynomial} if $p(\mathbf{0})>0$ and for every $\mathbf{0} \neq\mathbf{x} \in \mathbb{R}^n$, the univariate polynomial
$$p_{\mathbf{x}}(s) \coloneqq p(s \mathbf{x})$$
has only real roots. We call the set
$$\mathcal{R}(p) \coloneqq \{\mathbf{x} \in \mathbb{R}^n\mid  p_{\mathbf{x}} \text{ has no roots in } [0,1)\}$$
the \emph{rigidly convex set defined by} $p$.

We call a real zero polynomial smooth, if all $p_{\mathbf{x}}$ have only roots of multiplicity $1$. This implies that $\nabla p(\mathbf{x}) \neq 0$ whenever $p(\mathbf{x}) = 0$. Hence, the rigidly convex set $\mathcal{R}(p)$ of a smooth real zero polynomial is smooth.
\end{definition}

In the following we present some known results about real zero polynomials, necessary for the construction of the homotopy.

\begin{lemma}
\label{lem:RZ-properties}
Let $p\colon\mathbb{R}^n \to \mathbb{R}$ be a real zero polynomial. Then the following holds:
\begin{enumerate}[label=(\roman*)]
    \item $\mathcal{R}(p)$ is closed convex with $\mathbf{0}$ in its interior.
    \item If $p$ is smooth and $\mathcal{R}(p)$ is bounded, then $p$ is strictly quasi-concave on $\mathcal{R}(p)$.
\end{enumerate}
\end{lemma}
\begin{proof}
For $(i)$, see \cite{Ga59}, for $(ii)$, see \cite[Lemma 2.4]{Ne15}.
\end{proof}

We now discuss the example of determinantal polynomials. The corresponding rigidly convex sets in this case are spectrahedra.

\begin{example}
\label{ex:SDP_as_RZpoly}
Let $A(\mathbf{x}) = I_s + x_1 A_1 + \cdots + x_n A_n$ be a linear matrix function, and define the polynomial
    $$p(\mathbf{x}) \coloneqq \det(A(\mathbf{x})).$$
   Then $\lambda \neq 0$ is an eigenvalue of $A(\mathbf{x})$ with multiplicity $k$ if and only if $\frac{1}{1-\lambda}$ is a root of $p_{\mathbf{x}}$ with multiplicity $k$.
    This implies that $p$ is a real zero polynomial and $\mathcal{R}(p) = \mathcal{S}(A)$ is a spectrahedron.

Note that multiplicities $k \geq 2$ arise when the spectrahedron can be written as the intersection of other spectrahedra. If $C(\mathbf{x})$ is the linear matrix inequality generated by matrices $A_i \oplus B_i$, then
$$ \mathcal{S}(C) = \mathcal{S}(A) \cap \mathcal{S}(B). $$
Consequently, the corresponding real-zero polynomial satisfies
$$ p(\mathbf{x}) = \det(C(\mathbf{x})) = \det(A(\mathbf{x})) \cdot \det(B(\mathbf{x})). $$

If $ \mathbf{x} \in \partial \mathcal{S}(A) $ and $ \mathbf{x} \in \partial \mathcal{S}(B) $, then $ s = 1 $ is a root of $ s \mapsto p_{\mathbf{x}}(s) $ with multiplicity at least $k \geq 2$, reflecting the kink on the boundary of $ \mathcal{S}(C) $ at $ \mathbf{x} $.

\end{example}

\subsubsection{Smoothing of Real Zero Polynomials}

We introduce the \emph{smoothing operator} $\mathcal{F}_{\varepsilon}$ on the space of real-zero polynomials, which transforms each rigidly convex set into a smooth rigidly convex set that approximates the original. This smoothing procedure was initially developed for hyperbolic polynomials \cite{Nu68, Re04} (see also \cite{Ne12} for an overview). Since we are working in the affine setting, we introduce the smoothing operator directly at this level.

For this purpose, we introduce the \emph{Renegar derivative}.

\begin{definition}[Homogenization and the Renegar derivative]
Let $p\colon \mathbb{R}^n \to \mathbb{R}$ be a polynomial of degree $d$. We define the homogenization of $p$ by
$$\widetilde{p}(x_0,\mathbf{x}) \coloneqq x_0^d \cdot p\left(\frac{\mathbf{x}}{x_0}\right).$$
The \emph{Renegar derivative} of $p$ is defined as
$$(\partial_R p)(\mathbf{x}) \coloneqq \left(\frac{\partial}{\partial x_0} \widetilde{p}\right)(1,\mathbf{x}).$$
\end{definition}

\begin{definition}[Smoothing of Real-Zero Polynomials]
Let $p \colon \mathbb{R}^n \to \mathbb{R}$ be a polynomial. We define the linear operator $\mathcal{F}_{\varepsilon}$ as
$$\mathcal{F}_{\varepsilon}[p] \coloneqq p + \varepsilon \partial_R p,$$
Applied to real-zero polynomials, we denote it by the \emph{$\varepsilon$-smoothed version of} $p$.
\end{definition}

We now state some properties of the smoothing operator.

\begin{lemma}\label{lem:smoothing_properties}
Let $p\colon \mathbb{R}^n \to \mathbb{R}$ be a real zero polynomial. For $\varepsilon > 0$, the $\varepsilon$-smoothed version of $p$ has the following properties:
\begin{enumerate}[label=(\roman*)]
    \item\label{lem:sm_pr:cond1} $\mathcal{F}_{\varepsilon}[p]$ is a real zero polynomial, and hence $\mathcal{R}(\mathcal{F}_{\varepsilon}[p])$ is convex.
    \item\label{lem:sm_pr:cond2} If $\mathcal{R}(p)$ is bounded, then $\mathcal{R}(\mathcal{F}_{\varepsilon}[p])$ is bounded.
    \item\label{lem:sm_pr:cond3} If $\lambda$ is a root of multiplicity $k \geqslant 2$ of $p_{\mathbf{x}}$, then $\lambda$ is a root of multiplicity $k-1$ of $\mathcal{F}_{\varepsilon}[p]_{\mathbf{x}}$. All other roots of $\mathcal{F}_{\varepsilon}[p]_{\mathbf{x}}$ have multiplicity $1$.
\end{enumerate}
\end{lemma}
\begin{proof}
\ref{lem:sm_pr:cond1} and \ref{lem:sm_pr:cond3} follow from the fact that for every polynomial $h$ (with $h(\mathbf{0})\neq 0$) which has only real zeros, the polynomial $h + t h'$ also has only real zeros, with multiplicities as required. This can be deduced from general results in \cite{bor,Nu68} for example, or checked directly with standard calculus methods.

For \ref{lem:sm_pr:cond2}, let $q$ be the homogenization for $\mathcal{F}_{\varepsilon}[p]$. Then, we have that
$$ q(x_0, \mathbf{x}) = \widetilde{p}(x_0, \mathbf{x}) + \varepsilon x_0 \partial_{x_0} \widetilde{p}(x_0, \mathbf{x}).$$
$\mathcal{R}(p)$ is bounded, if for every $\mathbf{x} \in \mathbb{R}^n$, there exists some $x_0 > 0$ such that $\widetilde{p}(x_0, \mathbf{x}) = 0$. Similarly, $\mathcal{R}(\mathcal{F}_{\varepsilon}[p])$ is bounded, if for every $\mathbf{x} \in \mathbb{R}^n$, there is a $x_0 > 0$ such that $q(x_0, \mathbf{x}) = 0$.
Assuming the former, the latter is true, since if $h$ has a positive root, then the same is true for $h + \varepsilon t h'$, by the intermediate value theorem.
\end{proof}

\begin{remark}
By \cref{lem:smoothing_properties}, the $m$-fold composition of the smoothing operator
$$\mathcal{F}_{\varepsilon}^m(p) = p + \sum_{i=1}^m \varepsilon^i \partial_R^i p$$ reduces the roots of $p_{\mathbf{x}}$ with multiplicity $k \geqslant m + 1$ to roots with multiplicity $k-m$. All other roots of $\mathcal{F}_{\varepsilon}^m[p]_{\mathbf{x}}$ have multiplicity $1$.

Therefore, every rigidly-convex set $\mathcal{R}(p)$ can be smoothened by applying the smoothing operator $\mathcal{F}_{\varepsilon}$ sufficiently often. Moreover, if $\mathcal{R}(p)$ is bounded, the smoothed version also remains bounded.
\end{remark}

\subsubsection{A Regular Homotopy for Real Zero Polynomials}

We now define the homotopy for optimization on rigidly convex sets. 
Note that the product of two real zero polynomials $p,q$ is again a real zero polynomial, and we have 
\begin{equation}\label{eq:intersection}
    \mathcal{R}(p) \cap \mathcal{R}(q) = \mathcal{R}(p q).
\end{equation}

\begin{definition}\label{def:RZ-homotopy}
Let $p_0, p_1$ be real zero polynomials with bounded rigidly convex sets $\mathcal{R}(p_0)$ and $\mathcal{R}(p_1)$. Moreover, let $\varepsilon\colon [0,1] \to [0,\infty)$ be a smooth function with $\varepsilon(0) = \varepsilon(1) = 0$ and $\varepsilon(t) > 0$ for $t \in (0,1)$. We define the homotopies
\begin{enumerate}[label=(\roman*)]
    \item $\hat{p}_t(\mathbf{x}) \coloneqq p_0((1-t) \mathbf{x}) \cdot p_1(t \mathbf{x})$
    \item $p_t^{[k]}(\mathbf{x}) \coloneqq \mathcal{F}_{\varepsilon(t)}^k[\hat{p}_t](\mathbf{x}).$
\end{enumerate}
\end{definition}

\begin{theorem}\label{thm:RZ-homotopy}
The homotopies $\hat{p}_t$ and $p_t$ satisfy the following:
\begin{enumerate}[label=(\roman*)]
    \item\label{thm:RZ-homotopy:cond:1} If $\mathcal{R}(p_0)$ and $\mathcal{R}(p_1)$ are bounded, then $\mathcal{R}(\hat{p}_t)$ and $\mathcal{R}(p_t)$ are bounded.
    \item\label{thm:RZ-homotopy:cond:2} If $p_0$ and $p_1$ are smooth, then $p_t^{[1]}$ is smooth for every $t \in [0,1)$.
    \item\label{thm:RZ-homotopy:cond:3} If $p_0$ is smooth and $p_1$ has roots with multiplicity $k \geq 2$, then $p_t^{[k]}$ is smooth for every $t \in [0,1)$.
\end{enumerate}
\end{theorem}
\begin{proof}
\ref{thm:RZ-homotopy:cond:1} The boundedness of $\mathcal{R}(\hat{p}_t)$ follows immediately from \eqref{eq:intersection} and the boundedness of $\mathcal{R}(p_t)$ from \cref{lem:smoothing_properties} \ref{lem:sm_pr:cond2}.

\ref{thm:RZ-homotopy:cond:3} If $p_0$ and $p_1$ has roots with multiplicity $k$, we have that $s \mapsto \hat{p}_t(s \mathbf{x})$ has zeros of multiplicity at most $k+1,$ for every $t \in (0,1)$ and $0\neq \mathbf x \in \mathbb{R}^n$. Hence, applying the smoothing operator $\mathcal{F}_{\varepsilon(t)}^{k}$ gives rise to a smooth polynomial $p_t$ for every $t \in [0,1)$.

\ref{thm:RZ-homotopy:cond:2} Special case of \ref{thm:RZ-homotopy:cond:3} with $k = 1$.
\end{proof}

Note that $\hat{p}_t$ is in general not smooth, even if $p_0$ and $p_1$ are smooth (see for example \cref{fig:RZ-homotopy} for $\varepsilon = 0$ and $t = 0.6$). More specifically, the multiplicity of roots in $\hat{p}_t$ is in worst case the sum of the highest multiplicity arising in $p_0$ and $p_1$.

\begin{figure}
    \centering
    \includegraphics[scale=0.55]{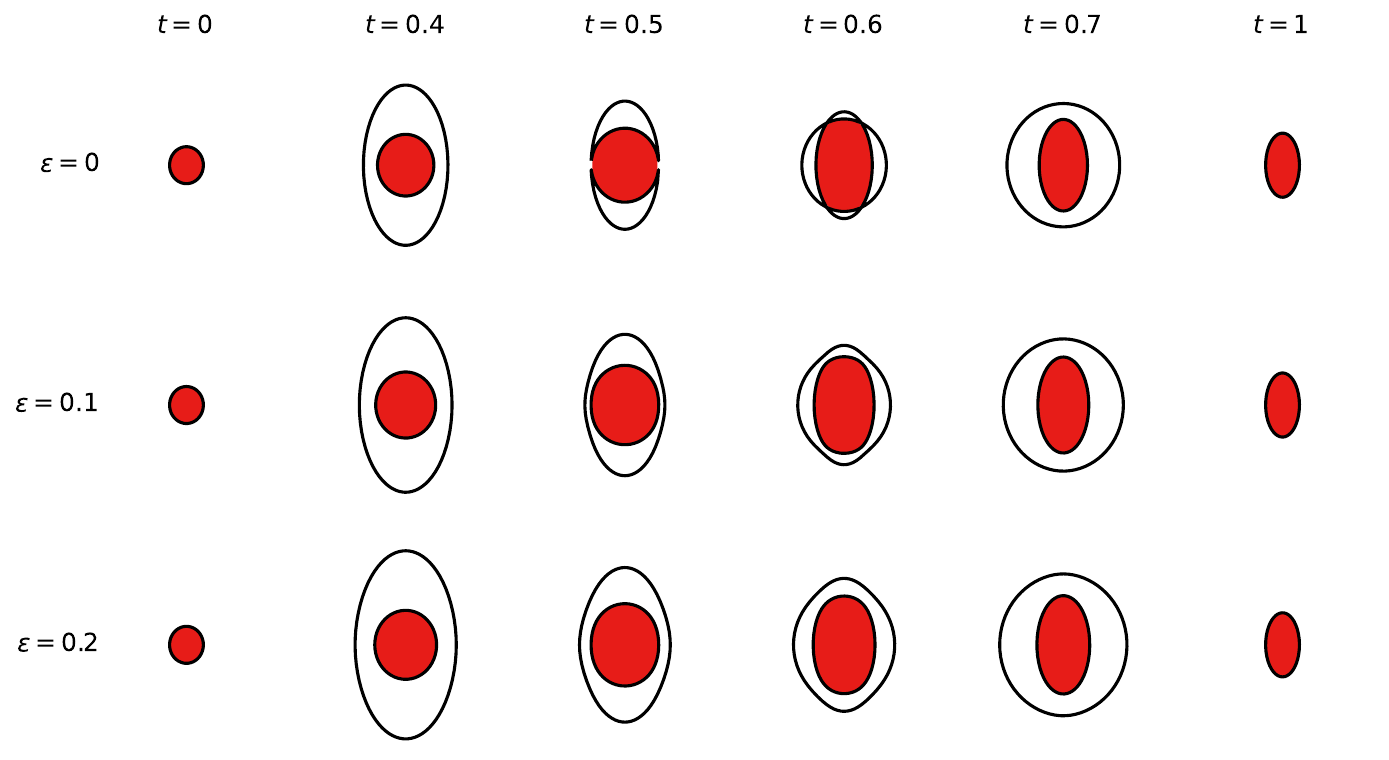}
    \caption{Examples of the homotopy within the space of real zero polynomials for different smoothing parameters. For $\varepsilon = 0$, the homotopy leads to non-smooth convex sets (see for example $t = 0.5$ and $t = 0.6$), as soon as $\varepsilon > 0$, the roots of multiplicity $2$ become roots of multiplicity $1$.}
    \label{fig:RZ-homotopy}
\end{figure}

\begin{theorem}\label{thm:main_RZ}
If $p_0$ is smooth, $p_1$ has roots with multiplicity $k$, and $\mathcal{R}(p_0), \mathcal{R}(p_1)$ are bounded, then the homotopy $p_t^{[k]}$ from \cref{def:RZ-homotopy} is regular, and therefore the induced differential equation \eqref{eq:diff_eq} has a unique solution.
\end{theorem}
\begin{proof}
Regularity follows from smoothness of the $p_t^{[k]}$. Indeed, if a point leaves the innermost ring of zeros of $p_t$ over time, this would imply a multiple zero of some $p_t^{[k]}$ on some line.
\end{proof}

For non-smooth target sets, the smoothing degree $k$ must be sufficiently large to ensure the uniqueness of the solution in \eqref{eq:diff_eq}. However, in practice, the numerical solution obtained from the differential equation often converges to the correct result even for smaller values of $k$ (see \cref{fig:lin_func}). This behavior may be attributed to two factors:
\begin{enumerate}[label=(\roman*)]
    \item The path of optimal points does not encounter any root of $\hat{p}_t$ with worst-case multiplicity, preserving the uniqueness of the solution.
    \item The numerical integrator, due to discretization in $t$, effectively bypasses the critical points.
\end{enumerate}

\textit{Remark}
A similar statement to\cref{thm:RZ-homotopy} can also be shown when dropping the smoothness-assumption for $\mathcal{R}(p_1)$. In this situation $s \mapsto \hat{p}_t(s \mathbf{a})$ can have zeros up to multiplicity $d \coloneqq \deg(p_1)$. Applying $\mathcal{F}_{\varepsilon(t)}$ $d$ times to $\hat{p},$ instead of only once, gives rise to a homotopy that satisfies the assumptions.

\begin{figure}
    \centering
    \includegraphics[scale=0.52]{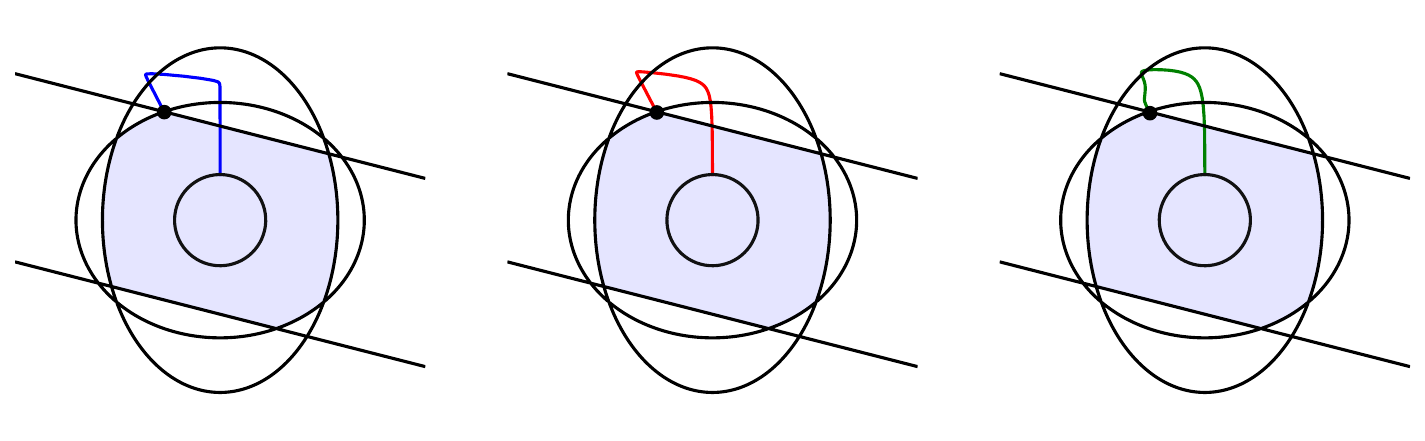}
    
    \label{fig:lin_func}
     
    \caption{Optimizing the linear function $f(x_1, x_2) = x_2$ over a (non-smooth) spectrahedron (blue). The three paths are obtained by using the the homotopies $\hat{p}_t$ (left, blue), $p_t^{[1]}$ (middle, red) and $p_{t}^{[2]}$ (right, green). While only $p^{[2]}$ is a smooth homotopy, numerically all paths lead to a correct solution. The blue path on the left has two singularities, the red path in the middle has only one singularity and the green path on the right has no singularity.}
\end{figure}

\subsection{Convex Optimization Problems with a Single Constraint}\label{ssec:singleConcaveConstraint}

In the following, we consider a convex optimization problem, where the non-empty feasibility set is defined by a single concave constraint. Without loss of generality, we assume that all feasible sets contain $\mathbf{0}$ in their interior.

More precisely, let $p_0,p_1\colon \mathbb{R}^n \to \mathbb{R}$ be continuously differentiable concave functions which are strictly positive at $\mathbf{0}$. The two corresponding convex sets are defined as
$$C_i := \{\mathbf{x} \in \mathbb{R}^n\mid  p_i(\mathbf{x}) \geqslant 0 \}.$$
By concavity of $p_i$, these are closed convex sets. Moreover, assume that the objective function $f\colon \mathbb{R}^n \to \mathbb{R}$ is linear.
We also assume to know 
the optimal value of
$$\max_{\mathbf{x} \in C_0} \  f(\mathbf{x}) $$
and we now want to compute the optimal value for the same problem on $C_1$.

For this purpose, consider the homotopy
$$p(t,\mathbf x)=p_t(\mathbf{x}) \coloneqq (1-t) p_0(\mathbf{x}) + t p_1(\mathbf{x})$$
for $t \in [0,1],$ which give rise to the sets 
$$C_t \coloneqq \{\mathbf{x} \in \mathbb{R}^n\mid  p_t(\mathbf{x}) \geqslant 0 \}.$$
Since $p_t$ is concave for every $t \in [0,1]$, every  $C_t$ is a closed convex set.
We now show that the homotopy satisfies the conditions for our method.

\begin{proposition}\label{prop:convexHomotopy}
Let $p_0, p_1$ be concave functions such that $p_0(\mathbf{0}), p_1(\mathbf{0}) > 0$. Then for every $t\in [0,1]$ we have:
\begin{enumerate}[label=(\roman*)]
    \item\label{convexHomotopy:cond:1} $\mathbf{x} \in \partial C_t$ if and only if $p_t(\mathbf{x}) = 0$.
    \item\label{convexHomotopy:cond:2} $C_t$ is smooth for every $t \in [0,1]$, i.e. $\nabla p_t \neq 0$ on $\partial C_t$.
    \item\label{convexHomotopy:cond:3} $C_0 \cap C_1 \subseteq C_t \subseteq C_0 \cup C_1$.
    \item\label{convexHomotopy:cond:4} If  $p_0$ is strictly concave, then $p_t$ is strictly concave for every $t \in [0,1)$.
\end{enumerate}
In particular, (iii) implies that $C_t$ has non-empty interior for every $t \in [0,1]$. Moreover, $C_t$ is bounded if $C_0, C_1$ are bounded.
\end{proposition}

\begin{proof}
\ref{convexHomotopy:cond:1} The only if direction is clear due to continuity. Now assume that $p_t(\mathbf{x}) = 0$. By concavity, we have that $\nabla p_t(\mathbf{x}) \neq 0$. But this implies that for every $\varepsilon$-neighborhood of $\mathbf{x}$ there are points $\mathbf{y}$ and $\mathbf{z}$ such that $p_t(\mathbf{y}) > 0$ and $p_t(\mathbf{z}) < 0$ which shows the statement.

\ref{convexHomotopy:cond:2} Using the same argument as in \ref{convexHomotopy:cond:1}, we have that for every $t \in [0,1]$, $\nabla p_t(\mathbf{x}) \neq 0$ for every $\mathbf{x} \in \partial C_t$, which shows the statement.

\ref{convexHomotopy:cond:3} Let $\mathbf{x} \in C_0 \cap C_1$, then $p_0(\mathbf{x}), p_1(\mathbf{x}) \geqslant 0$ and therefore $$(1-t) p_0(\mathbf{x}) + t p_1(\mathbf{x}) \geqslant 0,$$ which shows the one inclusion. For the second inclusion, assume that ${\mathbf{x} \notin C_0 \cup C_1}$. Then $p_0(\mathbf{x}), p_1(\mathbf{x}) < 0$ and hence
$$(1-t) p_0(\mathbf{x}) + t p_1(\mathbf{x}) < 0,$$ which shows the reverse inclusion.

\ref{convexHomotopy:cond:4} Immediate, since the sum of a concave and a strictly concave function is strictly concave.
\end{proof}

Now let  $p_0, p_1$ be three times continuously differentiable concave functions, with  $p_0$ strictly concave and $p_0(\mathbf{0}), p_1(\mathbf{0}) > 0$. Further assume the sets   $C_i=\{\mathbf x\mid p_i(\mathbf x)\geqslant 0\}$ are bounded for $i=0,1$. 
Then \cref{prop:convexHomotopy} together with \cref{prop:uniqueness} guarantees that the optimal solutions of the problems
$$ \max_{\mathbf{x} \in C_t} \   f(\mathbf{x})$$
can be found by solving the differential equation \eqref{eq:diff_eq}, where the initial value is given by the unique optimal value of the problem
$$ \max_{\mathbf{x} \in C_0} \   f(\mathbf{x}).$$
If $p_1$ is in addition strictly quasi-concave on the boundary of its nonnegativity set, then the solution of the optimization problem for $t = 1$ is unique.

\begin{example}\label{ex:convexOptimization problem}
Let $p_0(\mathbf{x}) = 1 - \Vert \mathbf{x} \Vert_2^2$ and $p_1$ a three times continuously differentiable concave function with $p_1(\mathbf{0}) > 0$ and $C_1$ bounded.

If $f(\mathbf{x}) = \mathbf{w}^t \mathbf{x}$ with $0\neq \mathbf w\in\mathbb R^n$, the optimal point of the problem for $t = 0$ is given by $\mathbf{a}_0 = \mathbf{w}/\Vert \mathbf{w} \Vert_2$. Since $\nabla f = \mathbf{w} \neq 0$ on $\mathbb{R}^n$, we have that the solution of  \eqref{eq:diff_eq} gives rise to the optimal points of the optimization problems \eqref{eq:C_optimization}.
    In \cref{fig:numeric_convex_homotopy} we show a numeric example for the path of optimal solutions for 
    $$p_1(\mathbf{x}) = 10 - \sum_{k=1}^4 \exp(\mathbf{b}_k^t \mathbf{x}),$$
    where $\mathbf{b}_1 = (\frac{1}{5},\frac{1}{5})^t$, $\mathbf{b}_2 = (-\frac{1}{10}, \frac{1}{5})^t$, $\mathbf{b}_3 = -\mathbf{b}_1$, and $\mathbf{b}_4 = (\frac{1}{10}, - \frac{1}{10})^t$, and the functional is given by $\mathbf{w} = (1,\frac{1}{2})^t$.
\end{example}

\begin{figure}
    \centering
         \includegraphics[scale=0.55]{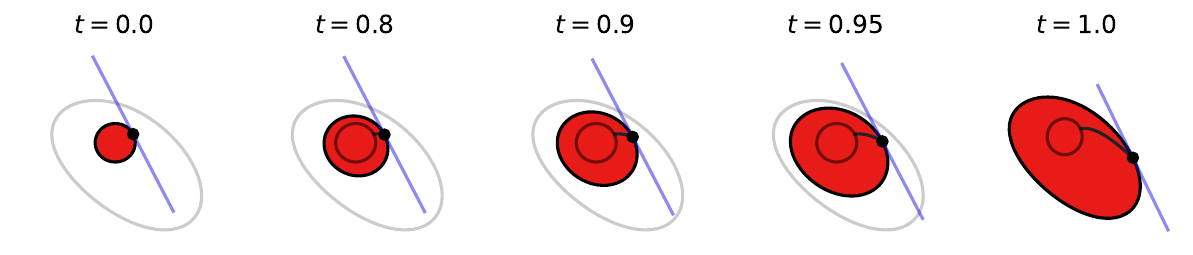}
     \caption{Homotopy of the convex optimization problem presented in \cref{ex:convexOptimization problem}. }
     \label{fig:numeric_convex_homotopy}
\end{figure}

\section{Numerical Examples and Benchmarks}
\label{sec:num_examples}

In the following, we present several numerical examples and compare the runtime of the homotopy method with known methods. This includes the following convex optimization problems:
\begin{itemize}
 \item Optimizing a linear function on a rigidly convex set generated by the elementary symmetric polynomials (\cref{ssec:hyperbolicProgramming}).
    \item Optimizing a linear function on a $k$-ellipsoid (\cref{ssec:kellipsoid}).
    \item Optimizing a linear function on a $p$-norm ball (\cref{ssec:pnorm_ball}).
    \item Geometric programming with a single constraint (\cref{ssec:geometricProgramming}).
   
\end{itemize}

Our starting set is always the unit ball, defined by the inequality $p_0(\mathbf x)=1-\Vert \mathbf x\Vert ^2\geqslant 0.$ Since $p_0$ is both strictly concave and a real zero polynomial, it fulfills all the requirements from above. And as already mentioned in \cref{ex:convexOptimization problem} above, computing the optimal point for a linear optimization problem on the unit ball in indeed trivial.

We emphasize that we did not focus on an optimized implementation of the homotopy method. In particular, the choice of the integrator for the differential equation is not tailored to the specific problem situation. Also, the implementation of the derivatives were done in the most basic way. Therefore, we believe that a more optimized implementation of the homotopy method might further  improve the runtime of the presented method significantly. 

The runtime-experiments were conducted on a standard laptop with 16GB RAM. The algorithms were run on Python, the semidefinite programs, the second-order conic programs, and the geometric programs were solved using the \textsc{Mosek} solver (version 9.3.20) \cite{mosek} via the \textsc{Cvxpy} package (version 1.3.2) \cite{Ag18, Di16}.
The differential equations for the homotopy method were solved via a standard predictor-corrector method using the explicit Runge-Kutta method of order $5(4)$. The errors for the optimal points were always fixed to be at most $10^{-5}$ for all methods.
\subsection{Hyperbolic Programming}\label{ssec:hyperbolicProgramming}

In the following, we benchmark the homotopy method for hyperbolic programs (\cref{ssec:homotopy_RZ}). Since other well-established solvers for hyperbolic programs seem not to exists, we compare our method to  semidefinite programming.  In particular, we consider the class of elementary symmetric polynomials $s_k$ (defined in \eqref{eq:elementary_symmetric}). While the elementary symmetric polynomial is of degree $k$, the best known way to represent its rigidly convex set needs matrices of a much larger size (see \cref{tab:matrixSize_elementary_symmetric}), which become intractable for SDP solvers already for small dimension $n$ and small parameters $k$.

For this purpose, recall the notion of hyperbolic polynomials. Let ${h\colon \mathbb{R}^{n+1} \to \mathbb{R}}$ be a homogeneous polynomial. We call $h$ \emph{hyperbolic in the direction} $\mathbf{e} \in \mathbb{R}^{n+1}$ if the univariate polynomial $t \mapsto h(\mathbf{y} - t \mathbf{e})$ has only real roots, for every $\mathbf{y} \in \mathbb{R}^{n+1}$. Hyperbolic polynomials give rise to a closed convex cone
$$\mathcal{C}_{\mathbf{e}}(h) \coloneqq \{\mathbf{y} \in \mathbb{R}^{n+1}\mid h(\mathbf{y} - t \mathbf{e}) \text{ has only nonnegative roots} \},$$
called the \emph{hyperbolicity cone} of $h$ in direction $\mathbf{e}$ (see  \cite{Ga59,Ne23,Re04} for further details). 

Note that hyperbolicity is precisely the homogenized version of the real zero property introduced in \cref{ssec:homotopy_RZ}.
\emph{Hyperbolic optimization} is the task of optimizing a linear objective function over an affine subspace of a hyperbolicity cone. This affine subspace is precisely a rigidly convex set of a real-zero polynomial (\cref{def:real_zero}). 

One particular class of hyperbolic polynomials are the \emph{elementary symmetric polynomials}. For $n \in \mathbb{N}$ and $k \in \{1, \ldots, n\}$, let
\begin{equation}\label{eq:elementary_symmetric}
s_{k}(y_0, \ldots, y_n) \coloneqq \sum_{0 \leqslant i_1 < \ldots < i_k \leqslant n} y_{i_1} \cdots y_{i_k}
\end{equation}
be the elementary symmetric polynomial of degree $k$ in $n+1$ variables.
All elementary symmetric polynomials are hyperbolic in direction $\mathbf{e} = (1, \ldots, 1)$. Moreover, we have that $s_k(\mathbf{e}) > 0$ for every $k \in \{1, \ldots, n\}$, and
$$ \mathbb{R}_{+}^{n+1} = \mathcal{C}_{\mathbf{e}}(s_{n+1}) \subseteq \mathcal{C}_{\mathbf{e}}(s_{n}) \subseteq \cdots \subseteq \mathcal{C}_{\mathbf{e}}(s_{2}) \subseteq \mathcal{C}_{\mathbf{e}}(s_{1}) = \left\{\mathbf{y} \in \mathbb{R}^{n+1}\mid \sum_{i=0}^{n} y_i \geqslant 0 \right\}.$$
In the following, we consider a bounded affine slice of the hyperbolicity cones $\mathcal{C}_{\mathbf{e}}(s_k)$. For this purpose, let $p_k$ be a dehomogenized version of $s_k$, defined as
\begin{equation}\label{eq:def_pk}
    p_k(x_1, \ldots, x_n) \coloneqq s_k\left(e_0 - \sum_{i=1}^{n} x_i, e_1 + x_1, \ldots, e_n + x_n \right).
\end{equation}
Then $p_k$ is a real-zero polynomial, since $p_k(\mathbf{0}) = s_k(\mathbf{e}) > 0$ and for every $s \in \mathbb{R} \setminus \{0\}$ and $\mathbf{x} \in \mathbb{R}^n \setminus \{0\}$, we have that
\begin{align*}
    p_k(\lambda \mathbf{x}) = s_k\left(e_0 - \sum_{i=1}^{n} \lambda x_i, e_1 + \lambda x_1, \ldots, e_n + \lambda x_n \right) = \lambda^k s_k\left( \frac{1}{\lambda} \mathbf{e} + \mathbf{y} \right),
\end{align*}
where $\mathbf{y} = (-\sum_{i=1}^n x_i, x_1, \ldots, x_n)$. Moreover, the rigidly convex set $\mathcal{R}(p_k)$ is bounded for $k\geqslant 2,$ see \cref{app:boundedness_sk}. Note that  $\mathcal{R}(p_k)$ corresponds (up to an isometry) to the set
\begin{equation}\label{eq:affinceSlice}
\mathcal{C}_{\mathbf{e}}(s_k) \cap \Big[\mathbf{e} + \{\mathbf{x} \in \mathbb{R}^{n+1}\mid \mathbf{e}^t \mathbf{x} = 0\} \Big].
\end{equation}

In the following, we will benchmark the problem
\begin{equation}\label{eq:symmetric_program}
        \max_{\mathbf{x} \in \mathcal{R}(p_k)} \ f(\mathbf{x}) 
\end{equation}
where $f\colon \mathbb{R}^n \to \mathbb{R}$ is a linear function.
We will use two different procedures:
\begin{enumerate}[label=(\roman*)]
    \item Applying the homotopy method for hyperbolic polynomials (\cref{ssec:homotopy_RZ}).
    \item\label{symmetric_poly_method:2} Applying an SDP interior point solver to the best known spectrahedral representation of $\mathcal{R}(p_k)$.
\end{enumerate}
Specifically, we will compare the runtimes of these algorithms for the situation where $k = n-1$, as well as for the situation where $k = \lfloor \frac{n+1}{2}\rfloor$.

Regarding \ref{symmetric_poly_method:2}, we use the best known spectrahedral representation of $\mathcal{C}_{\mathbf{e}}(p_k)$ introduced in \cite{Br14}. According to the construction, the matrix size of this representation corresponds to the number of strings of length at most $k-1$ with distinct characters from an alphabet of size $n+1$. In other words, the matrix size $s(k,n)$ is given by
$$s(k,n) = \sum_{i = 0}^{k-1} \binom{n+1}{i} \cdot i! \geqslant \left(\frac{n+1}{k-1}\right)^{k-1} \cdot (k-1)!$$
where we have used the lower bound $\binom{n}{i} \geqslant (n/i)^i$ for $i = k-1$. In particular, if $k \leqslant n$ and $k \sim n$, then $s(k,n) \sim c^k \cdot k!$ which makes the spectrahedral representation untractable already for small values $n$ and $k$. Values of $s(k,n)$ for the benchmarked situations are presented in \cref{tab:matrixSize_elementary_symmetric}.

\begin{table}
\small
    \centering
    \begin{tabular}{c | C{1.3cm} | C{1.3cm} | C{1.3cm} | C{1.3cm} | C{1.3cm} | C{1.3cm}}
        & $n = 3$ & $n = 4$ & $n = 5$ & $n = 6$ & $n = 7$ & $n = 8$\\
        \hline
        $k = n-1$ & $5$ & $26$ & $157$ & $\mathbf{1100}$ & $\mathbf{8801}$ & $\mathbf{79210}$ \\
        \hline
        $k = \lfloor \frac{n+1}{2} \rfloor$ & $5$ & $6$ & $37$ & $50$ & $\mathbf{401}$ & $\mathbf{5861}$
    \end{tabular}
    \caption{Matrix size of the spectrahedral representation of  $\mathcal{R}(p_k)$ for different dimensions $n$. The bold cases denote the matrix sizes, where the solver did not reach a conclusion (memory error).}
    \label{tab:matrixSize_elementary_symmetric}
\end{table}

\begin{figure}
    \centering
    \begin{subfigure}[t]{0.47\textwidth}
    \centering
        \includegraphics[scale=0.48]{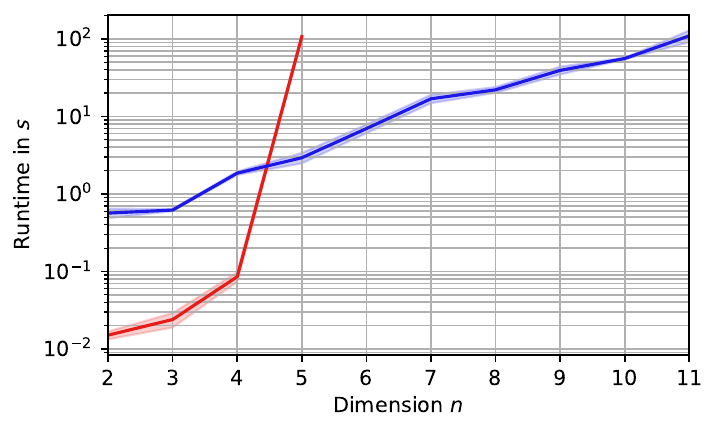}
        \subcaption{$k = n-1$}
    \end{subfigure}
    \hspace*{\fill}
    \begin{subfigure}[t]{0.47\textwidth}
    \centering
        \includegraphics[scale=0.48]{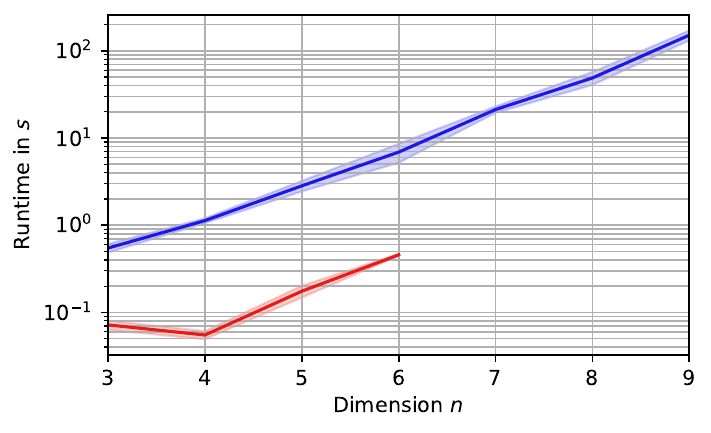}
        \subcaption{$k = \lfloor \frac{n+1}{2} \rfloor$}
    \end{subfigure}
    \caption{Runtimes of the homotopy method (blue) vs. the SDP solver (orange) applied to a random linear functional and $\mathcal{R}(p_k)$ for different dimensions $n$. For the SDP solver, we use the spectrahedral representation of $\mathcal{R}(p_k)$ from  \cite{Br14}. The SDP solver only reaches a conclusion for $n \leqslant 5$ and $n \leqslant 6$, respectively due to the growth of the matrix sizes (see \cref{tab:matrixSize_elementary_symmetric}).}
    \label{fig:symmetric_poly_cone_runtime}
\end{figure}

\cref{fig:symmetric_poly_cone_runtime} shows the runtimes of the homotopy method and the SDP solver for different values of $n$. While for small values of $n$, the SDP solver is much faster than the homotopy method, the SDP formulation is much slower or even intractable for larger values of $n$.

\subsection{Optimizing over the \texorpdfstring{$k$-ellipsoids}{k-ellipsoids}}\label{ssec:kellipsoid}

In the following, we benchmark the problem of optimizing a linear function on a $k$-ellipsoid. Given points $\mathbf{u}_1, \ldots \mathbf{u}_k \in \mathbb{R}^n$, we define the $k$-ellipsoid of dimension $n$ as the convex set
$$\mathcal{E}_n(\mathbf{u}_1, \ldots, \mathbf{u}_k,r) \coloneqq \left\{\mathbf{x} \in \mathbb{R}^n\mid \sum_{i = 1}^k \Vert \mathbf{x} - \mathbf{u}_i \Vert_2 \leqslant r \right\}.$$
The points $\mathbf{u}_1, \ldots, \mathbf{u}_k$ are called the \emph{focal points} of this ellipsoid.

We aim for solving the following problem:
\begin{equation}\label{eq:ellipsoid_optimization}
\begin{split}
    \max_{\mathbf{x} \in \mathbb{R}^n} & \  f(\mathbf{x}) \\
    \text{subject  to } & \ \mathbf{x} \in \mathcal{E}_n(\mathbf{u}_1, \ldots, \mathbf{u}_k,r)
\end{split}
\end{equation}
for different values of $n, k,$ and random choices of $\mathbf{u}_1, \ldots, \mathbf{u}_k$.

In the following, we will benchmark three different methods to solve \eqref{eq:ellipsoid_optimization}:
\begin{enumerate}[label=(\roman*)]
    \item\label{method:homotopy} Applying our homotopy method for a single concave constraint.
    \item\label{method:spec_rep} Using a spectrahedral representation of $\mathcal{E}_n(\mathbf{u}_1, \ldots, \mathbf{u}_k,r)$ to solve \eqref{eq:ellipsoid_optimization} via an SDP.
    \item\label{method:2nd_lift} Lifting the problem to a second-order cone program.
\end{enumerate}

Regarding \ref{method:homotopy}, \eqref{eq:ellipsoid_optimization} can be represented as
\begin{equation}\label{eq:kellipse_homotopy}
\begin{split}
    \max_{\mathbf{x} \in \mathbb{R}^n} \ & f(\mathbf{x})\\
    \text{subject to } & p(\mathbf{x}) \geqslant 0
\end{split}
\end{equation}
where
\begin{equation}
\label{eq:k-ellipse_equation}
    p(\mathbf{x}) \coloneqq r - \sum_{i = 1}^r \Vert \mathbf{x} - \mathbf{u}_i \Vert_2.
\end{equation}
Note that $p$ is not differentiable at the focal points $\mathbf{u}_i$. Even if $\mathbf{u}_i \notin \partial \mathcal{E}_n(\mathbf{u}_1, \ldots, \mathbf{u}_k,r)$, there are times $t$ where the functions $p_t(\mathbf{x})$, constructed from the homotopy in \cref{ssec:singleConcaveConstraint}, might not be differentiable and strictly concave. However, this happens for at most $k$ many different times. Therefore, the numerical integrator evaluates $p_t(\mathbf{x})$ on these problematic points with probability $0$. At all other points and times, the function is concave in our above sense, so the method will work with probability $1$. The homotopy method for a $3$-ellipse is shown in \cref{fig:ellipsoid_b}.

Regarding \ref{method:spec_rep}, in \cite{Ni07} it was shown that ellipsoids are spectrahedra. An explicit defining linear matrix inequality  is 
\begin{equation}\label{eq:kellipse_as_spectrahedron}
     \mathcal{M}(\mathbf{x}) \coloneqq r \cdot I_{(n+1)^k} - N_1(\mathbf{x}) \boxplus N_2(\mathbf{x}) \boxplus \cdots \boxplus N_k(\mathbf{x}) \succcurlyeq 0
\end{equation}
where
$$A_1 \boxplus A_2 \boxplus \cdots \boxplus A_k \coloneqq A_1 \otimes I \otimes \cdots \otimes I + I \otimes A_2 \otimes I \otimes \cdots \otimes I + \ldots + I \otimes \ldots \otimes I \otimes A_k$$
with $I$ being short-hand for the identity matrix of size $n+1$, and
$$N_j(\mathbf{x}) \coloneqq \begin{bmatrix} 0 & u_{j1} - x_1 & \ldots & u_{jn} - x_n \\ u_{j1} - x_1 & & & \\ \vdots & & 0 & \\ u_{jn}-x_n & & &\end{bmatrix}.$$
Note that \eqref{eq:kellipse_as_spectrahedron} is a linear matrix inequality whose matrix size is exponential in the number of focal points $k$. There is also no significantly  more efficient way of defining  the $k$-ellipse by a linear matrix inequality, since the degree of the defining polynomial also increases exponentially with $k$. More precisely, this degree is $2^k$ if $k$ is odd and $2^k - \binom{k}{k/2}$ if $k$ is even \cite{Ni07}.
\cref{fig:ellipsoid_a} shows the zero set of the real zero polynomial $\det(\mathcal{M}(\mathbf{x}))$ for a $3$-ellipse.

\begin{figure}
    \centering
    \begin{subfigure}[t]{0.45\textwidth}
        \centering
        \includegraphics[scale=0.7]{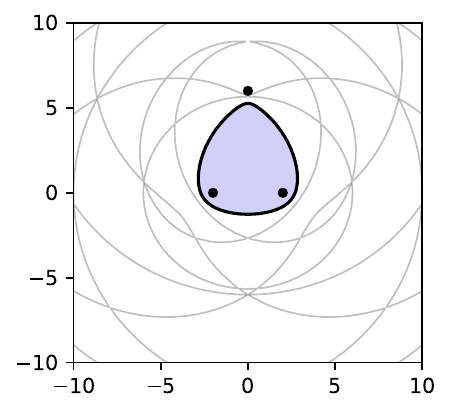}
        \subcaption{The zero set of the real zero polynomial $\det(\mathcal{M}(\mathbf{x}))$. The loop surrounding $\mathbf{0}$ bounds the $3$-ellipse.}
        \label{fig:ellipsoid_a}
    \end{subfigure}\hspace*{\fill}
    \begin{subfigure}[t]{0.45\textwidth}
        \centering
        \includegraphics[scale=0.7]{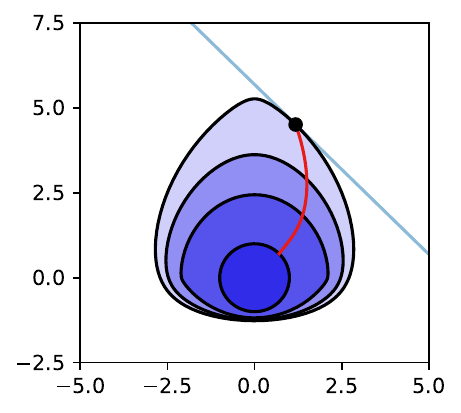}
        \subcaption{The homotopy method applied to the $3$-ellipse and the functional ${f(\mathbf{x}) = x_1 + x_2}$. The figure shows the sets at time $0$, $0.7$, $0.9$ and $1,$ and the path of optimal solutions.}
        \label{fig:ellipsoid_b}
    \end{subfigure}
    \caption{The $3$-ellipse generated by the points $\mathbf{u}_{1,2} = (\pm 2, 0)$ and $\mathbf{u}_3 = (6,0)$ with $r = 12$.}
    \label{fig:my_label}
\end{figure}

Therefore \eqref{eq:ellipsoid_optimization} is equivalent to the SDP
\begin{equation}\label{eq:kellipse_sdp}
\begin{split}
    \max_{\mathbf{x} \in \mathbb{R}^n} \ & f(\mathbf{x})\\
    \text{subject to } &  \mathcal{M}(\mathbf{x}) \succcurlyeq 0.
\end{split}
\end{equation}

Regarding \ref{method:2nd_lift}, \eqref{eq:ellipsoid_optimization} is equivalent to second-order cone program
\begin{equation*}
    \begin{split}
        \max_{\mathbf{x} \in \mathbb{R}^n,\ \mathbf{s} \in \mathbb{R}^k} \ & f(\mathbf{x})\\
        \text{subject to } \ & \sum_{i=1}^{k} s_i \leqslant r\\
        & \Vert \mathbf{x} - \mathbf{u}_i \Vert_2 \leqslant s_i \quad \text{ for all } i \in \{1, \ldots, k\}.
    \end{split}
\end{equation*}

The results of the benchmarks are summarized in \cref{tab:runtim_ellipsoid}. Since the dimension of the linear matrix inequality for the $k$-ellipse increases exponentially with the number of focal points $k$, the optimization problems becomes already intractable for small $k$. In contrast, the homotopy method remains tractable also for choices of large $k$, since the description only changes by adding terms to the defining equation \eqref{eq:k-ellipse_equation}.

\begin{table}
\small
    \centering
    \begin{tabular}{c | r | C{1.3cm} | C{1.3cm} | C{1.3cm} | C{1.3cm} | C{1.3cm} | C{1.3cm}}
        \multicolumn{2}{r|}{} & $k = 2$ & $k = 3$ & $k = 4$ & $k = 5$ & $k = 10$ & $k = 100$\\
        \hline
        \multirow{3}{*}{$n = 2$} & Homotopy & $0.02(1)$ & $0.02(1)$ & $0.06(4)$ & $0.04(1)$ & $0.09(4)$ & $1.5(4)$ \\
        & SDP & $0.03(1)$ & $0.07(1)$ & $3.3(4)$ & $\approx 1500$* & - & - \\
        & $2^{\text{nd}}$ order cone & $0.02(1)$ & $0.02(1)$ & $0.03(1)$ & $0.03(2)$ &$0.04(2)$ & $0.15(3)$ \\
        \hline
        \multirow{3}{*}{$n = 4$} & Homotopy & $0.05(2)$ & $0.05(1)$ & $0.07(3)$ & $0.10(2)$ &$0.21(5)$ & $4(1)$ \\
        & SDP & $0.06(1)$ & $23(2)$ & - & - & - & - \\
        & $2^{\text{nd}}$ order cone & $0.02(1)$ & $0.03(1)$ & $0.03(1)$ & $0.03(1)$ & $0.05(2)$ & $0.22(3)$  
    \end{tabular}
    \caption{Average runtime in seconds (and the standard deviation in brackets) of optimizing a linear function for random $k$-ellipsoids in dimension $n$ with $r = 2k$. The sample size is $100$.\\
    -\ : Solver did not reach a conclusion (Memory error)\\
    *\ : Only a single run.}
    \label{tab:runtim_ellipsoid}
\end{table}
\normalsize

\subsection{Optimizing over a \texorpdfstring{$p$-Norm}{p-Norm} Ball}\label{ssec:pnorm_ball}

In the following, we consider the optimization problem

\begin{equation}
\label{eq:Bp_optimization}
\max_{\mathbf{x} \in B_{p,r}} \  f(\mathbf{x}) \end{equation}
where $B_{p,r}$ is the ball of radius $r > 0$ around $\mathbf{0}$ with respect to the $p$-norm, i.e.
$$ B_{p,r} = \left\{\mathbf{x} \in \mathbb{R}^n\mid s_p(\mathbf{x}) \coloneqq r^p - \sum_{i = 1}^n |x_i|^p \geqslant 0 \right\}.$$
For simplicity, we restrict our analysis to the case $p = 8$.

Since $s_p$ is a convex function on $\mathbb{R}^n$, the homotopy introduced in \cref{ssec:singleConcaveConstraint} is a valid homotopy for this optimization problem. We will compare this method with the standard techniques via lifting \eqref{eq:Bp_optimization} to a semidefinite program or to a second-order cone program.

For the SDP representation, note that $B_{8,r}$ is not a spectrahedron (see \cite[Example 2.29]{Ne23} for a similar argument for $B_{4,r}$). However, it can be represented as the projection of a spectrahedron. This leads to the following SDP which is equivalent to \eqref{eq:Bp_optimization} for $p = 8$.

\begin{align*}
    \max_{\mathbf{x}, \mathbf{w}, \mathbf{u} \in \mathbb{R}^n} & \ f(\mathbf{x}) \\
    \text{subject to } & \ \begin{bmatrix} r^8 & u_1 & \cdots & u_n \\ u_1 & 1 & & \\ \vdots && \ddots & \\ u_n & & & 1\end{bmatrix} \succcurlyeq 0,\\
    & \ \begin{bmatrix} w_1 & x_1 \\ x_1 & 1 \end{bmatrix} \succcurlyeq 0, \ldots, \begin{bmatrix} w_n & x_n \\ x_n & 1 \end{bmatrix} \succcurlyeq 0, \\
    & \ \begin{bmatrix} u_1 & w_1 \\ w_1 & 1 \end{bmatrix} \succcurlyeq 0, \ldots, \begin{bmatrix} u_n & w_n \\ w_n & 1 \end{bmatrix} \succcurlyeq 0.
\end{align*}

Alternatively, $B_{8,r}$ can also be described via a projection of second-order cone constraints, i.e.\ constraints of the form
$$\Vert A \mathbf{x} + \mathbf{b} \Vert_2 \leqslant \mathbf{b}^t \mathbf{x} + d.$$
More precisely, \eqref{eq:Bp_optimization} for $p = 8$ is equivalent to

\begin{align*}
    \max_{\mathbf{x}, \mathbf{w}, \mathbf{u}, \mathbf{s} \in \mathbb{R}^n} & \ f(\mathbf{x}) \\
    \text{subject to } & \sum_{i=1}^n s_i \leqslant r^8 &\\
    & \left\Vert (x_i, w_i - 1)^t \right\Vert_2 \leqslant w_i + 1, &\text{ for } i = 1, \ldots, n \\ 
    &\left\Vert (w_i, u_i - 1)^t \right\Vert_2 \leqslant u_i + 1, &\text{ for } i = 1, \ldots, n \\ 
    &\left\Vert (u_i, s_i - 1)^t \right\Vert_2 \leqslant s_i + 1, &\text{ for } i = 1, \ldots, n.
\end{align*}

Hence, solving \eqref{eq:Bp_optimization} for $p = 8$ in dimension $n$ can be done in the three following ways:
\begin{enumerate}[label=(\roman*)]
  \item Using our homotopy method of \cref{ssec:singleConcaveConstraint}, i.e.\ solving a differential equation in $n$ dimensions.
    \item Solving an SDP in $\mathbb{R}^{3n}$ with matrix size $s = 5n + 1,$
    \item Solving a second-order cone program in $\mathbb{R}^{4n}$.
  
\end{enumerate}

Note that while the homotopy method applies uniformly to all values of $p > 1$, the SDP and second-order cone lifts highly depend on the choice of $p,$ which makes them more complex to implement. In particular, if $p \in \mathbb{Q}$, the dimensions of the lifts grow with the complexity of representing the fraction.

In \cref{fig:B8opt} we present the run-times of these three methods for different dimensions $n$, and a random linear functions. While the homotopy method is slower than the second-order cone lift, it outperforms the SDP lift for every dimension. We expect that the gap between second-order cone programming and the homotopy method shrinks for other values of $p$, in particular those, whose second-order cone lift needs a larger overhead.

\begin{figure}
\centering
    \includegraphics[scale=0.65]{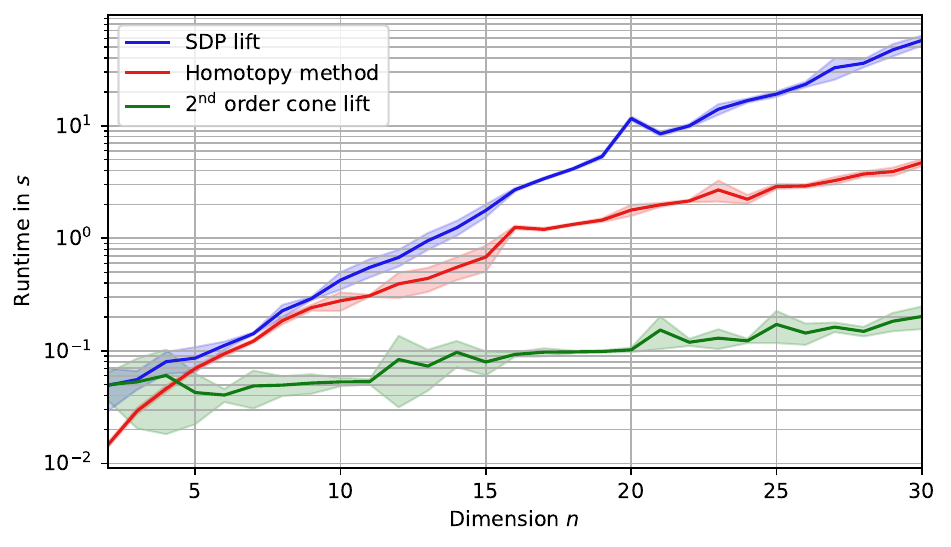}
    \caption{Solving the optimization problem \eqref{eq:Bp_optimization} for $p = 8$ via an SDP lift, the homotopy method, and via a second-order cone lift.}
    \label{fig:B8opt}
\end{figure}

Finally note that every convex semialgebraic set of the form 
$$\{\mathbf{x} \in \mathbb{R}^n\mid p(\mathbf{x}) \geqslant 0 \}$$
is a projected spectrahedron if $p$ is strictly quasi-concave on the set \cite{He07}. However, the matrix size $s$ and the space dimension are in general exponential in $n,$ which gives rise to even a larger gap between the runtimes of the SDP solver and the homotopy method in general.

Yet another alternative way is to compute the solutions of the Lagrange equations corresponding to the problem \eqref{eq:Bp_optimization} for $p = 8$, i.e.\ computing solutions of the system of polynomial equations:
\begin{align*}
    r^8 &= \sum_{i = 1}^{n} x_i^8 \\
    f(\mathbf{e}_i) &= - 8 \lambda x_i^7 \quad \text{ for } \ i = 1, \ldots, n
\end{align*}
According to Bezout's theorem, this system of equations has at most $8^{n+1}$ solutions. So the  established homotopy method from numerical algebraic geometry  has to track an exponential number of solution paths. Therefore, the runtime of the homotopy method for algebraic equations \cite{Ha17} scales exponentially in the dimension $n$. In contrast, our method tracks precisely one path, namely the path of optimal solutions.

\subsection{Geometric Programming with a Single Constraint}\label{ssec:geometricProgramming}

A geometric program is a optimization problem of the following form:
\begin{alignat*}{2}
        \min_{\mathbf{y} \in \mathbb{R}^n} \ & f_0(\mathbf{y}) & \\
        \text{subject to } & f_i(\mathbf{y}) \leqslant 1 & \textrm{ for } i = 1, \ldots, m\\
        & y_j > 0 & \textrm{ for } j = 1, \ldots, n
\end{alignat*}
where $f_0, \ldots, f_m$ are \emph{posynomial} functions, i.e.\
$$ f_i(\mathbf{y}) = \sum_{k=1}^{r} c_{k} y_1^{b_{k1}^{(i)}} \cdots y_n^{b_{kn}^{(i)}}$$
with $c_{k} \geqslant 0, b_{kj}^{(i)} \in \mathbb{R},$ and $r$ is the number of monomials involved in the constraint functions. By substituting $y_j = \exp(x_j)$, the geometric program translates into the convex optimization problem
\begin{equation}
    \begin{split}
        \min_{\mathbf{x} \in \mathbb{R}^n} \ & g_0(\mathbf{x}) \\
        \textrm{subject to } & g_i(\mathbf{x}) \leqslant 1 \textrm{ for } i = 1, \ldots, m
    \end{split}
\end{equation}
where
$$g_i(\mathbf{x}) \coloneqq \sum_{k}^r c_{k} \exp(\mathbf{b}_k^{(i)^t} \mathbf{x}).$$

In the following, we benchmark the subclass of geometric programs with $f_0$  a single monomial and only one constraint, i.e. $m = 1$. This is then equivalent to the problem
\begin{equation}
    \begin{split}
        \max_{\mathbf{x} \in \mathbb{R}^n} \ & f(\mathbf{x}) \\
        \textrm{subject to } & p(\mathbf{x}) \geqslant 0 
    \end{split}
\end{equation}
where $f$ is a linear function and
$$p(\mathbf{x}) = 1 - \sum_{k = 1}^r \exp(\mathbf{b}_k^t \mathbf{x} + a_k)$$
is a concave function. Therefore, this problem can be solved by our method using the homotopy in \cref{ssec:singleConcaveConstraint}.

In the benchmark, we  compare our homotopy method with the interior point method for exponential cones provided by the \textsc{Cvxpy} package (version 1.3.2) in Python for different choices of $r$ and $n$. The plots of the runtimes are shown in \cref{fig:exponential_cone_runtime}. As the plots show, the interior point method performs better than the homotopy method for these examples. However, we believe that the homotopy method can be improved by using more tailored solvers for the differential equation, or a different homotopy which is tailored to geometric programs. 

\begin{figure}
    \centering
    \begin{subfigure}[t]{0.45\textwidth}
    \centering
        \includegraphics[scale=0.47]{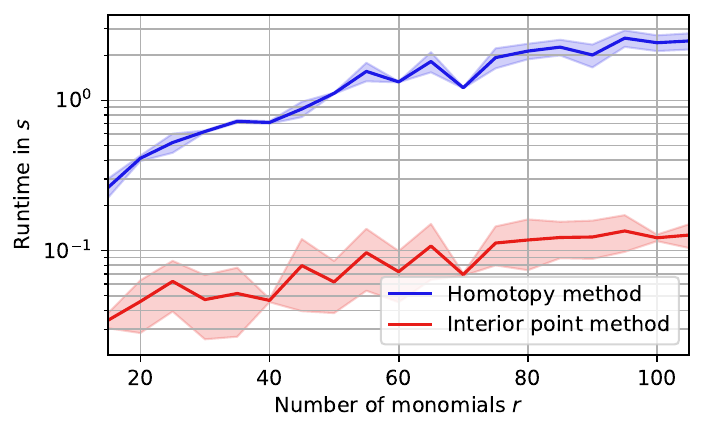}
        \subcaption{Varying the number of monomials $r$ for fixed dimension $n = 4$.}
    \end{subfigure}
    \hspace*{\fill}
    \begin{subfigure}[t]{0.45\textwidth}
    \centering
        \includegraphics[scale=0.47]{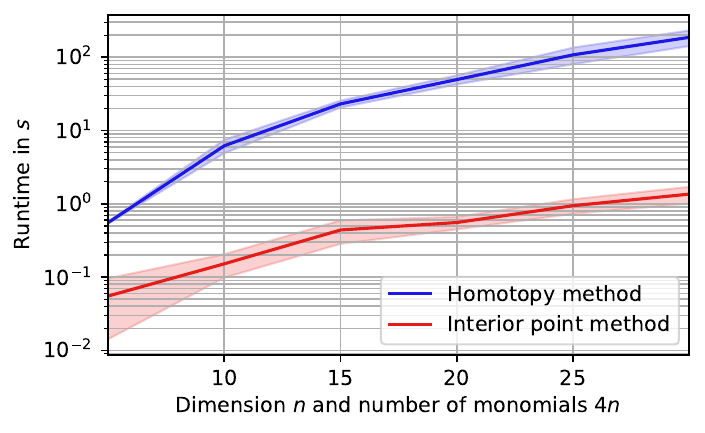}
        \subcaption{Varying the dimension $n$. For the number of monomials we choose $r = 4n$ to guarantee that the feasable set is bounded with high probability.}
    \end{subfigure}
    \caption{Runtimes of the homotopy method vs. the interior point method applied to random geometric programs. For simplicity, we choose $a_k = 0$ and $\mathbf{b}_k$ to be random vectors with entries in $[-1,1]$. Moreover, we choose $C \in [r, 2r]$ randomly to guarantee that $0$ is in the set of feasable points. For each configuration we have a sample size of $10$. The line shows the mean runtime together with the standard deviation.}
    \label{fig:exponential_cone_runtime}
\end{figure}

% \subsection{Higher-Order Smoothing for non-smooth target SDPs}

% In the following, we will consider a semidefinite program such whose feasible set is explicitly non-smooth. Specifically, we consider the problem

% \begin{align*}
%     \max_{\mathbf{x} \in \mathbb{R}^n} \ & f(\mathbf{x}) \\
%     \textrm{subject to } & \mathbf {x} \in \mathcal{S}(A) \cap \mathcal{S}(B) = \mathcal{S}(A \oplus B)
% \end{align*}

% where $A$ and $B$ are random matrices.

% Since the defining real zero polynomial
% $$p(\mathbf{x}) = \det(A \oplus B(\mathbf{x}))$$
% has has multiplicities of degree $2$ at most (see \cref{ex:SDP_as_RZpoly}), it suffices to consider the smoothing operator of degree $2$
% $$\mathcal{F}_{\varepsilon(t)}^2[p_t] = p_t + \varepsilon(t) \partial_R p_t + \varepsilon^2(t) \partial_R^2 p_t$$
% in order to guarantee that the homotopy is smooth (see \cref{thm:RZ-homotopy}).

\section{Conclusion and Open Questions}
\label{sec:outlook}

In this paper, we have introduced a method to solve certain classes of convex optimization problems, by introducing a homotopy between a trivial optimization problem and the target problem, and following the path of optimal solutions along the homotopy of optimization problems. In \cref{sec:diff_eq} we have shown that the path of optimal solutions is described by the  non-linear differential equation \eqref{eq:diff_eq}, which has a unique solution under certain assumptions on the homotopy (\cref{prop:uniqueness}).
We have shown that this approach can be applied to examples including semidefinite  and hyperbolic programs (\cref{ssec:homotopy_RZ}) as well as convex optimization problems with a single constraint (\cref{ssec:singleConcaveConstraint}). Moreover, we have shown that our method leads to a significant speed-up in certain numerical examples, most notably hyperbolic programming (\cref{ssec:hyperbolicProgramming}).

The general applicability of this method begs the question under which assumptions we can also solve more general convex optimization problems with the homotopic approach. This includes for example:
\begin{itemize}
    \item solving convex optimization problems with multiple constraints, or
    \item solving optimization problems with quasi-convex constraints.
\end{itemize}

It also remains unclear if the boundedness requirement of the feasible sets $C_t$ can be relaxed. While boundedness limits hyperbolic programming only to affine slices of hyperbolicity cones, it is crucial for preventing solution divergence, even though smoothing also works in the unbounded case.

Moreover, we leave an exact runtime analysis of this method for further work. This would include for example a comparison of the complexity of our method with interior-point methods applied to SDPs. Finally, since there are multiple canonical choices of homotopies for certain optimization classes 
this suggests to study whether there exist alternative choices of homotopies that lead to computationally more efficient results.

%%================================================

\section*{Acknowledgments}

This research was funded in part by the Austrian Science Fund (FWF) [doi:\href{https://www.doi.org/10.55776/P33122}{10.55776/P33122}]. AK further acknowledges funding of the Austrian Academy of Sciences (\"OAW) through the DOC scholarship 26547. Moreover, AK wants to thank Hamza Fawzi,  and Ois\'{i}n Faust for fruitful discussions. 

\bibliographystyle{abbrvnat}
\bibliography{references}

\appendix
\addtocontents{toc}{\protect\setcounter{tocdepth}{0}}

\section{Boundedness of the Rigidly Convex Sets of Symmetric Polynomials}
\label{app:boundedness_sk}

In the following, we prove that $\mathcal{R}(p_k)$ (see \eqref{eq:def_pk}) is bounded for $k \geqslant 2$.

\begin{proposition}
$\mathcal{R}(p_k)$ is bounded for $k \geqslant 2$.
\end{proposition}
\begin{proof}
Since $\mathcal{R}(p_k)$ is an affine slice of the hyperbolicity cones $\mathcal{C}_{\mathbf{e}}(s_k)$ (see \eqref{eq:affinceSlice}), we have that
$$\mathcal{R}(p_{k+1}) \subseteq \mathcal{R}(p_k),$$
by \cite[Remark 2.75]{Ne23}.

Therefore it suffices to show that $\mathcal{R}(p_2)$ is bounded. We show that for every $\mathbf 0\neq \mathbf{a} \in \mathbb{R}^n$, the polynomial $\lambda \mapsto p_2(\lambda \mathbf{a})$ has roots with different sign.
Elementary manipulations lead to 
$$p_2(\mathbf{x}) = \frac{n (n+1)}{2} - \sum_{1\leqslant i\leqslant j\leqslant n}^{n} x_ix_j.$$ 
Hence, for every $\mathbf{a} \in \mathbb{R}^{n}$, we have that
$$p_2(\lambda \cdot \mathbf{a}) = \frac{n (n+1)}{2} - b \lambda^2 $$
with
$$ b \coloneqq \sum_{1\leqslant i\leqslant j\leqslant n} a_ia_j > 0 .$$
Thus the two roots must have different sign, which  implies that $\mathcal{R}(p_2)$ is bounded.
\end{proof}

\end{document}